\documentclass{amsart}

\usepackage{amsmath}
\usepackage{amsfonts}
\usepackage{amssymb}
\usepackage{cite}

\numberwithin{equation}{section}

\newcommand{\abs}[1]{\left\vert#1\right\vert}
\newcommand{\norm}[1]{\left\Vert#1\right\Vert}
\newcommand{\paren}[1]{\left(#1\right)}
\newcommand{\bracket}[1]{\left[#1\right]}
\newcommand{\set}[1]{\left\{#1\right\}}

\newcommand{\N}{\mathbb{N}}
\newcommand{\R}{\mathbb{R}}
\newcommand{\T}{\mathbb{T}}

\newcommand{\C}{\mathbb{C}}
\newcommand{\moll}{\mathcal{J}_{\varepsilon}}
\newcommand{\ueps}{u_{\varepsilon}}
\newcommand{\vareps}{{\varepsilon}}
\newcommand{\cc}{\mathrm{c.c.}}

\newtheorem{theorem}{Theorem}[section]
\newtheorem{lemma}[theorem]{Lemma}
\newtheorem{cor}[theorem]{Corollary}

\title[Generalized Derivative Nonlinear Schr\"{o}dinger Equations]{Local Existence Theory for Derivative Nonlinear Schr\"{o}dinger Equations with 
Non-Integer Power Nonlinearities}

\author{David M. Ambrose}
\address{Department of Mathematics, Drexel University, Philadelphia, PA 19104, USA}
\thanks{DMA gratefully acknowledges support from the National Science Foundation through grants DMS-1008387 and DMS-1016267.}
\author{Gideon Simpson}
\address{Department of Mathematics, Drexel University, Philadelphia, PA 19104, USA}

\date{\today}

\begin{document}
\maketitle

\begin{abstract} We study a derivative nonlinear Schr\"{o}dinger
  equation, allowing non-integer powers in the nonlinearity,
  $\abs{u}^{2\sigma} u_x$.  Making careful use of the energy method,
  we are able to establish short-time existence of solutions with
  initial data in the energy space, $H^1$.  For more regular initial
  data, we establish not just existence of solutions, but also
  well-posedness of the initial value problem.  These results hold for
  real-valued $\sigma\geq 1,$ while prior existence results in the
  literature require integer-valued $\sigma$ or $\sigma$ sufficiently
  large ($\sigma \geq 5/2$), or use higher-regularity function spaces.
\end{abstract}

\section{Introduction}

We consider the following generalization of the derivative nonlinear
Schr\"odinger equation (DNLS),
\begin{equation}
  \label{e:gdnls}
  iu_{t}+i\abs{u}^{2\sigma}u_{x}+u_{xx} = 0,\quad u:\T\times \R \to \C
\end{equation}
with $\sigma\geq 1$, a real number.  This equation, gDNLS, was
recently studied in \cite{Liu:2013cq,Liu:2013ej} for both the
properties of its solitons and its potential for singularity
formation.

The original derivative nonlinear Schr\"odinger equation, was
written as
\begin{equation}
  \label{e:dnls}
  i v_t + i (\abs{v}^{2} v)_x + v_{xx}=0,
\end{equation}
and appeared as a long-wavelength approximation for Alfv{\'e}n waves
in plasma physics \cite{Sulem:1999kx}.  Under a non-degenerate gauge
transformation, \eqref{e:dnls} becomes \eqref{e:gdnls} with
$\sigma=1$, the cubic case.  With $\sigma=1,$ \eqref{e:gdnls} has also
appeared as a model for self-steepening optical pulses
\cite{Moses:2007vv}.

The cubic version of \eqref{e:gdnls} ({\it i.e.}, the case $\sigma=1$) has
received significant attention, with early well-posedness results due
to Tsutsumi and Fukuda, followed by work by Hayashi and Ozawa, in $H^1$
and higher regularity spaces
\cite{Hayashi:1992wl,Hayashi:1993vj,Hayashi:1994un,Hayashi:1994vp,Ozawa:1996uj,Tsutsumi:1980uy}.
More recently Colliander {\it et al.} and Gr\"unrock and Herr have
treated the problem in $H^s$ spaces with $s<1$
\cite{Colliander:2001ve,Colliander:2002ve,Grunrock:2008go,Herr:2006ev}.

Generalizing to $\sigma>1$, the analysis is incomplete.  In
\cite{Hao:2007ez}, local well-posedness was proven for \eqref{e:gdnls}
in $H^{1/2}$, but only for $\sigma \geq 5/2$.  The theorems of
\cite{Kenig:1993wr,Kenig:1998wr} could be directly applied to problems
with $\sigma \in \N$, but these would also give results in
higher-regularity spaces.  While the paper \cite{Tsutsumi:1980uy} only
treats the cubic problem, it is mentioned there that non-integer
powers in the nonlinearity may be treated by the same method.  This is
left vague, however, as it is stated that the nonlinearity must be
``smooth.''  In any case, the result of \cite{Tsutsumi:1980uy} for the
cubic equation is in the space $H^{1.5+}$ which is more regular than
the energy space.  Therefore, none of these works treat the same
parameter regime as in the present study, since we allow real-valued
$\sigma\geq 1$ and we demonstrate existence in $H^{1},$ the energy
space.

In two recent works, \cite{Liu:2013cq,Liu:2013ej}, \eqref{e:gdnls} was
considered with $1<\sigma<2$.  Simulations presented in these works
used smooth data, on a periodic domain.  To be well justified, it is
thus desirable to have a high regularity well-posedness result, with
solutions in, say, $H^2$.  Additionally, in \cite{Liu:2013cq}, the
stability of the soliton solutions of gDNLS was examined in $H^1$.
The results are conditional on the existence of weak solutions, $u\in
C(0,T;H^1(\R))$, such that
\begin{equation}
  \label{e:weak_soln}
  \frac{d}{dt}\left \langle u(t), \phi \right \rangle = \left \langle
    \frac{\delta \mathcal{H}}{\delta \bar u}, i \phi \right \rangle, \quad \phi \in H^{-1}
\end{equation}
where $\mathcal{H}$ is the Hamiltonian defined below in
\eqref{e:hamiltonian}.  This formulation is motivated by the framework
of Grillakis, Shatah, and Strauss,
\cite{Grillakis:1987hj,Grillakis:1990jv}, whose methods were applied
in \cite{Liu:2013cq}.  Thus, it is of interest to develop an $H^1$ and
higher regularity theory for \eqref{e:gdnls} with more flexibility on
$\sigma$.  For a general initial condition $u_{0}\in H^{1},$ we do not
establish the existence of a solution $u\in C(0,T;H^{1}),$ but instead
we find $u\in L^{\infty}(0,T;H^{1})\cap C(0,T;H^{s}),$ for any $s<1.$
However, by studying more regular initial data, $u_{0}\in H^{2},$ we
find solutions in $u\in C(0,T;H^{2}).$ These solutions are also in
$C(0,T;H^{s})$ for all $s<2,$ and thus are in $C(0,T;H^{1}).$

The motivation for obtaining results in $H^1$ is because this is the
energy space of the problem.  Indeed, \eqref{e:gdnls} formally
conserves
\begin{equation}
  \label{e:hamiltonian}
  \mathcal{H}=\int |u_{x}|^{2}+\frac{1}{(\sigma+1)^{2}}
  \bar{u}^{\sigma+1}D_{x} u^{\sigma+1}\ dx, \quad D_{x}\equiv \frac{1}{i}\partial_{x},
\end{equation}
which is well-defined for $H^1$ functions.  Though we will not make
use of it, \eqref{e:gdnls} can then be written as
\begin{equation}\nonumber
  u_t = -i \frac{\delta \mathcal{H}}{\delta \bar u}.
\end{equation}
The equation also conserves mass
\begin{equation}
  \label{e:mass}
  \mathcal{M} = \int \abs{u}^2 dx,
\end{equation}
and momentum
\begin{equation}
  \label{e:momentum}
  \mathcal{P} = \int - \frac{1}{2} \bar u D_x u.
\end{equation}

Studying the problem on $\T = [0, 2\pi)$, we seek mild solutions $u
\in L^\infty(0,T;H^1)$ such that
\begin{equation}
  \label{e:mild_soln}
  u= e^{i\partial_{x}^{2}t}u_{0}
  -\int_{0}^{t}e^{i\partial_{x}^{2}(t-s)}|u|^{2\sigma}u_{x}\ ds.
\end{equation}
Our first main result is:
\begin{theorem}\label{mainTheorem}
  Let $\sigma\geq1$ be given.  Let $u_{0}\in H^{1}.$ There exists
  $T>0$ and $u\in L^\infty(0,T;H^1)$ such that $u$ is a mild solution
  of (\ref{e:gdnls}), i.e., $u$ satisfies (\ref{e:mild_soln}).
  Furthermore, for all $s$ such that $0\leq s <1,$ $u\in
  C(0,T;H^{s}).$
\end{theorem}

In addition to proving the existence of solutions for the initial
value problem in $H^{1}$ as described in Theorem \ref{mainTheorem}, we
also are able to demonstrate some properties of solutions.  In
particular, we will discuss the conservation of the invariants for
these solutions, and the regularity with regard to time.

Our existence proof follows the energy method.  We first introduce a
mollified evolution equation, for which we can apply the Picard
theorem to get existence of solutions.  We prove an energy estimate
for the solutions of the mollified problem, which demonstrates that
the solutions cannot blow up until a certain time, with this time
being independent of the mollification parameter.  We conclude that
the solutions of the mollified problem exist on a common time
interval, and this enables us to take the limit of the solutions as
the mollification parameter vanishes.  Finally, we demonstrate that
the limit solves the non-mollified evolution equation (\ref{e:gdnls}).
The energy method is, generally speaking, easier to apply at higher
regularity.  In the present setting, in which the data is only in
$H^{1},$ this final step is fairly delicate, and requires significant
effort.

If we allow for higher regularity, taking the initial data from
$H^{2}$ instead, we can make additional estimates, and we can then
prove more.  This is the content of our second theorem:
\begin{theorem}\label{t:h2} Let $\sigma\geq 1$ be given.  Let
  $u_{0}\in H^{2}.$ There exists $T>0$ and a unique $u \in C(0,T;H^2)$
  satisfying \eqref{e:mild_soln}.  Furthermore, the solution depends
  continuously upon the initial conditions in the $H^s$ norm for all
  $0\leq s<2.$
\end{theorem}

In Section 2 below, we will state some elementary results that will be
useful throughout the rest of the text.  In Section 3, we establish
short-time existence of solutions when the initial data is in $H^{1},$
and we study the conserved quantities for these solutions.  We then
establish our $H^2$ result in Section 4.  We conclude with some
remarks and conjectures in Section 5.

\section{Preliminary Results}

Our results are based on the energy method, which requires us first to
mollify \eqref{e:gdnls}, then to prove existence of solutions for the
mollified problem, and finally to let the mollification parameter
vanish.  In this section, we introduce the mollifiers and establish
some preliminary properties needed for our main results. For
$\varepsilon>0$ we introduce a mollifier $\mathcal{J}_{\varepsilon}.$
We choose this so that it is a projection, {\it i.e.},
$\mathcal{J}_{\varepsilon}=\mathcal{J}_{\varepsilon}^{2};$ since we
are studying the spatially periodic case, it makes sense to let this
be the projection onto modes with wavenumber at most
$\frac{1}{\varepsilon}.$ We let $\mathcal{F}$ represent the (periodic)
Fourier transform, so that
\begin{equation}
  \mathcal{F}(\mathcal{J}_{\varepsilon}g-g)(k) = \begin{cases}
    0 & \abs{k}\leq \lceil\frac{1}{\varepsilon}\rceil,\\
    -\mathcal{F}(g)(k) & \abs{k} > \lceil\frac{1}{\varepsilon}\rceil.
  \end{cases}
\end{equation}

We use the following mollifier inequality, which holds for any
$\varepsilon>0$ and any $f\in L^{2}.$
\begin{equation}\label{mollifierInequality}
  \|\mathcal{J}_{\varepsilon}f\|_{L^{2}}\leq\|f\|_{L^{2}}.\end{equation}
Another inequality we will use is, for $s\geq0,$
\begin{equation}\label{mollifierGains}
  \|\mathcal{J}_{\varepsilon}f\|_{H^{s}}\leq \frac{c}{\varepsilon^{s}}\|f\|_{L^{2}}.
\end{equation}
Also, we can take the following limit, for $g\in H^{m}$ with $m\geq
0:$
\begin{equation}\label{mollifierLimit}
  \lim_{\varepsilon\rightarrow0^{+}}\|\mathcal{J}_{\varepsilon}g-g\|_{H^{m}}=0.
\end{equation}
These properties of the mollifier follow directly from Plancherel's
Theorem.

The next result concerns the interplay of this spatial mollifier with
a space-time norm:
\begin{lemma}\label{plancherelTonelli}
  Let $\tau>0$ and $m\geq0$ be given, and let $g$ be an element of
  $L^{2}(0,\tau;H^{m}).$ Then $\mathcal{J}_{\eta}g$ converges to $g$
  in this space, as $\eta\rightarrow0^{+}.$
\end{lemma}

\begin{proof}We consider the norm of $\mathcal{J}_{\eta}g-g:$
$$\|\mathcal{J}_{\eta}g-g\|_{L^{2}(0,\tau;H^{m})}^{2}=\int_{0}^{\tau}\|\mathcal{J}_{\eta}g-g\|_{H^{m}}^{2}
\ ds
=\int_{0}^{\tau}\int_{0}^{2\pi}|\Lambda^{m}(\mathcal{J}_{\eta}g-g)|^{2}+|\mathcal{J}_{\eta}g-g|^{2}
\ dxds,$$ where $\Lambda$ is the operator with symbol
$\mathcal{F}\Lambda(k)=|k|.$ We use Plancherel's Theorem:
$$\|\mathcal{J}_{\eta}g-g\|_{L^{2}(0,\tau;H^{m})}^{2}=\int_{0}^{\tau}
\sum_{k=-\infty}^{\infty}(1+|k|^{2m})|\mathcal{F}(\mathcal{J}_{\eta}g-g)(k)|^{2}\
ds.$$ The definition of $\mathcal{J}_{\eta}$ implies that for $|k|<N,$
where $N\sim\frac{1}{\eta},$ these Fourier coefficients are equal to
zero.  Also, for $|k|\geq N,$ we have $\mathcal{J}_{\eta}g(k)=0.$ We
can thus rewrite the sum:
$$\|\mathcal{J}_{\eta}g-g\|_{L^{2}(0,\tau;H^{m})}^{2}=\int_{0}^{\tau}
\sum_{|k|\geq N}(1+|k|^{2m})|\mathcal{F}(g)(k)|^{2}\ ds.$$ By
Tonelli's Theorem, we can exchange the order of the integral and the
sum:
\begin{equation}\label{almostLemma}
  \|\mathcal{J}_{\eta}g-g\|_{L^{2}(0,\tau;H^{m})}^{2}=
  \sum_{|k|\geq N}\int_{0}^{\tau}(1+|k|^{2m})|\mathcal{F}(g)(k)|^{2}\ ds.\end{equation}
Since $g\in L^{2}(0,\tau;H^{m}),$ we know that
$$\int_{0}^{\tau}\|g\|_{H^{m}}^{2}\ ds = \sum_{k=-\infty}^{\infty}\int_{0}^{\tau}
(1+|k|^{2m})|\mathcal{F}(g)(k)|^{2}\ ds$$ is finite.  The right-hand
side of (\ref{almostLemma}) is therefore the sum of the tails of a
convergent series; as such, it goes to zero as $N\rightarrow\infty,$
which is the same as saying as $\eta$ vanishes.  This completes the
proof of the lemma.
\end{proof}

We will also make use of the elementary Sobolev interpolation lemma:
\begin{lemma}\label{interp}
  Let $m\geq0$ and $\ell\geq m$ be given.  Let $f\in H^{\ell}$ be
  given.  Then, the following inequality holds:
  \begin{equation}
    \|f\|_{H^m}\leq c \|f\|_{H^\ell}^{m/\ell}\|f\|_{L^2}^{1-m/\ell}.\end{equation}
\end{lemma}
The proof of this lemma can be found in \cite{ambroseThesis}, among
other places.

\section{Existence theory in $H^{1}$}
\label{s:well_posedness}

Our mollified evolution equation is:
\begin{equation}\label{uMollEquation}
  u_{\varepsilon,t}=-\mathcal{J}_{\varepsilon}\left(
    |\mathcal{J}_{\varepsilon}u_{\varepsilon}|^{2\sigma}\mathcal{J}_{\varepsilon}u_{\varepsilon,x}\right)
  +i\mathcal{J}_{\varepsilon}u_{\varepsilon,xx}.
\end{equation}
There is a slight abuse of notation in the last term of the above
expression in that we imagine $u_{\varepsilon}\in H^1$, so its second
derivative may not be well defined.  However, we treat
\[
\mathcal{J}_{\varepsilon}u_{\varepsilon,xx}
= \partial_{x}^2 \paren{\moll u_{\varepsilon}},
\]
which resolves any ambiguity. Our data is unmollified, and denoted
$$u_{\varepsilon}(x,0)=u_{0}(x)\in H^1(\T).$$
We consider the case $\sigma\geq1$ and seek existence of solutions in
the space $H^{1}.$

We introduce the notation
$$u_{\varepsilon,t}=\mathcal{J}_{\varepsilon}F_{\varepsilon}(u_{\varepsilon}),$$
so that $$F_{\varepsilon}(u_{\varepsilon})=
-|\mathcal{J}_{\varepsilon}u_{\varepsilon}|^{2\sigma}\mathcal{J}_{\varepsilon}u_{\varepsilon,x}
+i\mathcal{J}_{\varepsilon}u_{\varepsilon,xx}.$$ Notice that we have
used here the property that
$\mathcal{J}_{\varepsilon}=\mathcal{J}_{\varepsilon}^{2}.$

\noindent{\bf Step 1:} Existence for a very short time.

To show that \eqref{uMollEquation} has a local in time solution, we
must prove that $\mathcal{J}_{\varepsilon}F_{\varepsilon}$ is locally
Lipschitz continuous on $H^{1}.$ To begin, we use
\eqref{mollifierGains} as follows:
$$\|\mathcal{J}_{\varepsilon}F_{\varepsilon}(f)-\mathcal{J}_{\varepsilon}F_{\varepsilon}(g)\|_{H^{1}}
\leq
\frac{c}{\varepsilon}\|F_{\varepsilon}(f)-F_{\varepsilon}(g)\|_{L^{2}}.$$
After adding and subtracting, we apply the triangle inequality to
obtain:
\begin{multline}
  \|F_{\varepsilon}(f)-F_{\varepsilon}(g)\|_{L^{2}}\leq\||\mathcal{J}_{\varepsilon}f|^{2\sigma}
  (\mathcal{J}_{\varepsilon}(f_{x}-g_{x}))\|_{L^{2}}\\
  +\|(\mathcal{J}_{\varepsilon}g_{x})(|\mathcal{J}_{\varepsilon}f|^{2\sigma}
  -|\mathcal{J}_{\varepsilon}g|^{2\sigma})\|_{L^{2}}
  +\|\mathcal{J}_{\varepsilon}(f_{xx}-g_{xx})\|_{L^{2}}.
\end{multline}
The first and third terms on the right-hand side can clearly be
bounded by $\|f-g\|_{H^{1}};$ for the third term this again makes use
of (\ref{mollifierGains}).  For the second term on the right-hand
side, it can also be bounded by $\|f-g\|_{H^{1}}$ (in fact it can be
bounded by $\|f-g\|_{L^{2}}$) using the fact that $|z|^{2\sigma}$ is a
Lipschitz continuous function. By the Picard Theorem for ODEs on a Banach space, this implies that
there exists $T_{\varepsilon}>0$ and $u_{\varepsilon}\in
C^{1}((-T_{\varepsilon},T_{\varepsilon});H^{1})$
which is a solution of the initial value problem.\\

\noindent{\bf Step 2:} A uniform time interval.

Next, we would like to establish that the solutions are uniformly
bounded with respect to $\varepsilon;$ this will allow us to extend
the interval of existence to be independent of $\varepsilon.$

We will shortly introduce our energy functional.  The $H^{1}$
norm of the solution, $u_{\varepsilon},$ will be controlled by the
energy.  Before introducing it, we make a few estimates.

\begin{lemma}\label{lowerOrderEstimate}
  For any real number $r\geq 1,$ there exist constants $c>0$ and
  $p\geq1$ such that
  \begin{align*}
    \frac{d}{dt}\int_{0}^{2\pi}|u_{\varepsilon}|^{2r}\ dx&\leq c(1+\|u_{\varepsilon}\|_{H^{1}}^{2})^{p},\\
    \frac{d}{dt}\int_{0}^{2\pi}|\mathcal{J}_{\varepsilon}u_{\varepsilon}|^{2r}\
    dx &\leq c(1+\|u_{\varepsilon}\|_{H^{1}}^{2})^{p}.
  \end{align*}
\end{lemma}
We emphasize that in the above inequalities, $c$ and $p$ are
independent of $\varepsilon.$
\begin{proof}
  This follows from direct calculation.  For the first expression,
  \begin{equation*}
    \begin{split}
      \frac{d}{dt}\int |u_{\varepsilon}|^{2r} & = \int \bar \ueps^{r} \ueps^{r-1} u_{\vareps,t} + \cc\\
      &=\int -\bar \ueps^{r} \ueps^{r-1}\abs{\moll
        \ueps}^{2\sigma}\moll u_{\vareps, x} + i \bar\ueps^{r}
      \ueps^{r-1} \moll u_{\vareps, xx} + \cc.
    \end{split}
  \end{equation*}
  We integrate by parts in the second integral on the right-hand side,
  and we apply a derivative:
  \begin{equation*}
    \begin{split}
      \frac{d}{dt}\int |u_{\varepsilon}|^{2r}
      & = \int -\bar \ueps^{r} \ueps^{r-1}\abs{\moll \ueps}^{2\sigma}\moll u_{\vareps, x}  - i\left ( \bar \ueps^{r} \ueps^{r-1}\right)_x \moll u_{\vareps, x}  + \cc\\
      & = \int -\bar \ueps^{r} \ueps^{r-1}\abs{\moll \ueps}^{2\sigma}\moll u_{\vareps, x}  - i r\abs{\ueps}^{2r-2} \bar u_{\vareps,x} \moll u_{\vareps, x}  \\
      &\quad \quad - i (r-1) \bar \ueps^{r} \ueps^{r-2} u_{\vareps,x}
      \moll u_{\vareps, x} + \cc.
    \end{split}\end{equation*}
  We are then able to bound all of these terms:
  \begin{equation*}\begin{split}
      \frac{d}{dt}\int |u_{\varepsilon}|^{2r}
      &\leq 2\norm{\ueps}_{L^\infty}^{2r-2}\norm{\moll \ueps}_{L^\infty}^{2\sigma} \norm{\ueps}_{L^2}\norm{\moll u_{\vareps, x}}_{L^2} \\
      &\quad \quad + 2r\norm{\ueps}_{L^\infty}^{2r-2} \norm{u_{\vareps,x}}_{L^2}\norm{\moll u_{\vareps, x}}_{L^2}\\
      &\quad \quad + 2 (r-1) \norm{\ueps}_{L^\infty}^{2r-2} \norm{u_{\vareps,x}}_{L^2}\norm{\moll u_{\vareps,x}}_{L^2}\\
      & \leq \norm{\ueps}_{H^1}^{2r + 2\sigma} + (4r-2) \norm{\ueps}_{H^1}^{2r}\\
      & \leq (4r-1) \left(1+ \norm{\ueps}_{H^1}^2 \right)^{r +
        \sigma}.
    \end{split}
  \end{equation*}
  The calculation for the second inequality is similar.
\end{proof}

We introduce an approximate version of the Hamiltonian,
\eqref{e:hamiltonian}:
\begin{equation}\label{mollHamiltonian}
  \mathcal{H}_{\varepsilon}[f]=\int_{0}^{2\pi}
  |f_{x}|^{2}
  +\frac{1}{(\sigma+1)^{2}}(\mathcal{J}_{\varepsilon}\bar{f})^{\sigma+1}D_x(\mathcal{J}_{\varepsilon}f)^{\sigma+1}\ dx.
\end{equation}
This is conserved for solutions of the mollified evolution; this is
the content of the following lemma.
\begin{lemma}\label{hamiltonianEstimate} The evolution conserves
  $\mathcal{H}_{\varepsilon}[u_{\varepsilon}].$ That is,
$$\frac{d\mathcal{H}_{\varepsilon}[u_{\varepsilon}]}{dt}=0.$$
\end{lemma}

\begin{proof}
  We take the time derivative of
  $\mathcal{H}_{\varepsilon}[u_{\varepsilon}]:$
  \begin{multline}\nonumber
    \frac{d\mathcal{H}_{\varepsilon}[u_{\varepsilon}]}{dt}=\int_{0}^{2\pi}
    u_{\varepsilon,xt}\bar{u}_{\varepsilon,x}+u_{\varepsilon,x}\bar{u}_{\varepsilon,xt}
    -\frac{1}{\sigma+1}(\mathcal{J}_{\varepsilon}\bar{u}_{\varepsilon})^{\sigma}(\mathcal{J}_{\varepsilon}\bar{u}_{\varepsilon,t})
    i\partial_{x}(\mathcal{J}_{\varepsilon}u_{\varepsilon})^{\sigma+1}\\
    -\frac{1}{\sigma+1}(\mathcal{J}_{\varepsilon}\bar{u}_{\varepsilon})^{\sigma+1}i\partial_{x}\bracket{(\mathcal{J}_{\varepsilon}u_{\varepsilon})^{\sigma}
      (\mathcal{J}_{\varepsilon}u_{\varepsilon,t})}\ dx.
  \end{multline}
  We integrate by parts and so on, yielding the following:
  \begin{multline}\nonumber
    \frac{d\mathcal{H}_{\varepsilon}[u_{\varepsilon}]}{dt}=\int_{0}^{2\pi}
    -u_{\varepsilon,t}\moll\bar{u}_{\varepsilon,xx}-\moll
    u_{\varepsilon,xx}\bar{u}_{\varepsilon,t}
    -(\mathcal{J}_{\varepsilon}\bar{u}_{\varepsilon})^{\sigma}(\mathcal{J}_{\varepsilon}\bar{u}_{\varepsilon,t})
    i(\mathcal{J}_{\varepsilon}u_{\varepsilon})^{\sigma}(\mathcal{J}_{\varepsilon}u_{\varepsilon,x})\\
    +(\mathcal{J}_{\varepsilon}\bar{u}_{\varepsilon})^{\sigma}(\mathcal{J}_{\varepsilon}\bar{u}_{\varepsilon,x})
    i(\mathcal{J}_{\varepsilon}u_{\varepsilon})^{\sigma}(\mathcal{J}_{\varepsilon}u_{\varepsilon,t})\
    dx.
  \end{multline}
  In the above expression, we have used that
  $\mathcal{J}_{\varepsilon}u_{\varepsilon,t}=u_{\varepsilon,t}$,
  allowing us to integrate by parts and still have a quantity in
  $L^2$.  Using this property, and that
  $\mathcal{J}_{\varepsilon}^{2}=\mathcal{J}_{\varepsilon}$, the
  expression becomes:
  \begin{equation}\label{toEstimateHamiltonian}
    \begin{split}
      \frac{d\mathcal{H}_{\varepsilon}[u_{\varepsilon}]}{dt}=\int_{0}^{2\pi}&
      -u_{\varepsilon,t}\moll\bar{u}_{\varepsilon,xx}-\moll
      u_{\varepsilon,xx}\bar{u}_{\varepsilon,t}\\
      &\quad
      -i|\mathcal{J}_{\varepsilon}u_{\varepsilon}|^{2\sigma}(\mathcal{J}_{\varepsilon}u_{\varepsilon,x})\bar{u}_{\varepsilon,t}
      +i|\mathcal{J}_{\varepsilon}u_{\varepsilon}|^{2\sigma}(\mathcal{J}_{\varepsilon}\bar{u}_{\varepsilon,x})u_{\varepsilon,t}\
      dx.
    \end{split}
  \end{equation}
  Next, we plug in from the evolution equation \eqref{uMollEquation}.
  We write the result as
$$\frac{d\mathcal{H}_{\varepsilon}[u_{\varepsilon}]}{dt}=A_{1}+A_{2}+A_{3}+A_{4},$$
where each of these corresponds to one of the terms in the integrand
in \eqref{toEstimateHamiltonian}.  We spell out each of these:
\begin{align*}
  A_{1}&=\int_{0}^{2\pi}\moll\bar{u}_{\varepsilon,xx}\mathcal{J}_{\varepsilon}\left(
    |\mathcal{J}_{\varepsilon}u_{\varepsilon}|^{2\sigma}\mathcal{J}_{\varepsilon}u_{\varepsilon,x}\right)\
  dx -
  \int_{0}^{2\pi}i\moll\bar{u}_{\varepsilon,xx}\left(\mathcal{J}_{\varepsilon}u_{\varepsilon,xx}\right)\
  dx,\\
  A_{2}&=\int_{0}^{2\pi}\moll
  u_{\varepsilon,xx}\mathcal{J}_{\varepsilon}\left(
    |\mathcal{J}_{\varepsilon}u_{\varepsilon}|^{2\sigma}\mathcal{J}_{\varepsilon}\bar{u}_{\varepsilon,x}\right)\
  dx +\int_{0}^{2\pi}i\moll
  u_{\varepsilon,xx}\left(\mathcal{J}_{\varepsilon}\bar{u}_{\varepsilon,xx}\right)\
  dx,\\
  \begin{split}
    A_{3}&=\int_{0}^{2\pi}i|\mathcal{J}_{\varepsilon}u_{\varepsilon}|^{2\sigma}(\mathcal{J}_{\varepsilon}u_{\varepsilon,x})
    \mathcal{J}_{\varepsilon}\left(|\mathcal{J}_{\varepsilon}u_{\varepsilon}|^{2\sigma}(\mathcal{J}_{\varepsilon}\bar{u}_{\varepsilon,x})\right)\ dx \\
    &-\int_{0}^{2\pi}|\mathcal{J}_{\varepsilon}u_{\varepsilon}|^{2\sigma}(\mathcal{J}_{\varepsilon}u_{\varepsilon,x})
    (\mathcal{J}_{\varepsilon}\bar{u}_{\varepsilon,xx})\ dx,
  \end{split}\\
  \begin{split}
    A_{4} &=
    -\int_{0}^{2\pi}i|\mathcal{J}_{\varepsilon}u_{\varepsilon}|^{2\sigma}(\mathcal{J}_{\varepsilon}\bar{u}_{\varepsilon,x})
    \mathcal{J}_{\varepsilon}\left(|\mathcal{J}_{\varepsilon}u_{\varepsilon}|^{2\sigma}\mathcal{J}_{\varepsilon}u_{\varepsilon,x}\right)\
    dx
    \\
    &-\int_{0}^{2\pi}|\mathcal{J}_{\varepsilon}u_{\varepsilon}|^{2\sigma}(\mathcal{J}_{\varepsilon}\bar{u}_{\varepsilon,x})
    (\mathcal{J}_{\varepsilon}u_{\varepsilon,xx})\ dx.
  \end{split}
\end{align*}
Adding these, and repeatedly using the fact that
$\mathcal{J}_{\varepsilon}$ is self-adjoint, we see that they all
cancel.
Thus, $$\frac{d\mathcal{H}_{\varepsilon}[u_{\varepsilon}]}{dt}=0.$$

\end{proof}

We are almost in a position to define our energy functional.  This
requires, however, first establishing a relationship between the
Hamiltonian and the $H^{1}$ norm.

\begin{lemma}\label{normEnergyLemma}
  There exists a constant $\bar{c}>0$, independent of $\varepsilon$,
  such that
  \begin{equation}\label{normEnergyInequality}
    \frac{1}{2}\|u_{\varepsilon}\|_{H^{1}}^{2}
    \leq \mathcal{H}_{\varepsilon}[u_{\varepsilon}]+\int_{0}^{2\pi} \frac{1}{2}|u_{\varepsilon}|^{2}
    + \bar{c}|\mathcal{J}_{\varepsilon}u_{\varepsilon}|^{4\sigma+2}\ dx.\end{equation}
\end{lemma}

\begin{proof}
  We use the following (standard) definition of the square of the
  $H^{1}$ norm:
$$\|u_{\varepsilon}\|_{H^{2}}^{1}=\int_{0}^{2\pi}|u_{\varepsilon}|^{2}+|u_{\varepsilon,x}|^{2}\ dx.$$
Using this with (\ref{mollHamiltonian}) and
(\ref{normEnergyInequality}), we see that we are attempting to find
$\bar{c}>0$ such that the following inequality is true:
\begin{equation}\nonumber
  \int_{0}^{2\pi}\frac{1}{(\sigma+1)^{2}}(\mathcal{J}_{\varepsilon}\bar{u}_{\varepsilon})^{\sigma+1}
  i\partial_{x}(\mathcal{J}_{\varepsilon}u_{\varepsilon})^{\sigma+1}\ dx
  \leq \int_{0}^{2\pi}\frac{1}{2}|u_{\varepsilon,x}|^{2}+\bar{c}|\mathcal{J}_{\varepsilon}u_{\varepsilon}|^{4\sigma+2}\ dx.
\end{equation}
We apply the derivative in the integrand on the left-hand side:
\begin{equation}\nonumber
  \frac{1}{(\sigma+1)^{2}}(\mathcal{J}_{\varepsilon}\bar{u}_{\varepsilon})^{\sigma+1}
  i\partial_{x}(\mathcal{J}_{\varepsilon}u_{\varepsilon})^{\sigma+1}
  =\frac{i}{\sigma+1}(\mathcal{J}_{\varepsilon}\bar{u}_{\varepsilon})^{\sigma+1}
  (\mathcal{J}_{\varepsilon}u_{\varepsilon})^{\sigma}\mathcal{J}_{\varepsilon}u_{\varepsilon,x}.
\end{equation}
We bound this with its absolute value, and we also use Young's
Inequality:
\begin{equation*}
\begin{split}
  \frac{1}{(\sigma+1)^{2}}(\mathcal{J}_{\varepsilon}\bar{u}_{\varepsilon})^{\sigma+1}
  i\partial_{x}(\mathcal{J}_{\varepsilon}u_{\varepsilon})^{\sigma+1}
  &\leq\left|\frac{i}{\sigma+1}(\mathcal{J}_{\varepsilon}\bar{u}_{\varepsilon})^{\sigma+1}
    (\mathcal{J}_{\varepsilon}u_{\varepsilon})^{\sigma}\mathcal{J}_{\varepsilon}u_{\varepsilon,x}\right|\\
  &\leq
  \frac{1}{2(\sigma+1)^{2}}\left|\mathcal{J}_{\varepsilon}u_{\varepsilon}\right|^{4\sigma+2}+\frac{1}{2}|\mathcal{J}_{\varepsilon}u_{\varepsilon,x}|^{2}.
\end{split}
\end{equation*}
Integrating, and using the inequality (\ref{mollifierInequality}), the
proof is complete.  Note that the constant $\bar{c}$ is given by
$\bar{c}=\frac{1}{2(\sigma+1)^{2}}.$
\end{proof}

We are now able to define the energy; it is given by the right-hand
side of \eqref{normEnergyInequality}.  We define
\begin{equation}
  \label{e:energy}
  \mathcal{E}_\varepsilon=\mathcal{H}_{\varepsilon}[u_{\varepsilon}]
  +\int_{0}^{2\pi}\frac{1}{2}|u_{\varepsilon}|^{2} +
  \bar{c}|\mathcal{J}_{\varepsilon}u_{\varepsilon}|^{4\sigma+2}\ dx.
\end{equation}
Using Lemma \ref{lowerOrderEstimate}, Lemma \ref{hamiltonianEstimate},
and Lemma \ref{normEnergyLemma}, we are able to conclude
$$\frac{d\mathcal{E}_\varepsilon}{dt}\leq c(1+\mathcal{E}_\varepsilon)^{p}.$$

This implies that the energy cannot blow up arbitrarily fast.
Combining this with the Continuation Theorem for ODEs on a Banach
space, we conclude that there exists $T>0$ such that for all
$\varepsilon>0,$ the solution $u_{\varepsilon}$ exists on the time
interval $[0,T].$ The solutions $u_{\varepsilon}$ are therefore
uniformly bounded in the space $C([0,T];H^{1}).$

\

\noindent{\bf Step 3:} Passage to the limit as $\varepsilon\rightarrow
0.$

The result of this step is the existence of $u\in
L^{\infty}([0,T];H^{1}),$ which is the limit of a subsequence of
$u_{\varepsilon}.$ We prove this by using the Aubin-Lions Lemma (see
Lemma 8.4 of \cite{constantinFoias}).  In particular, our family of
approximate solutions, $u_{\varepsilon},$ is uniformly bounded in
$L^{\infty}([0,T];H^{1}),$ and it is therefore uniformly bounded in
$L^{2}(0,T;H^{1}).$ Furthermore, inspection of the equation
(\ref{uMollEquation}) shows that the family $u_{\varepsilon,t}$ is
uniformly bounded in $L^{2}(0,T;H^{-1}).$ Since $H^{1}([0,2\pi])$ is
compactly embedded in $L^{2}([0,2\pi]),$ which is in turn continuously
embedded in $H^{-1}([0,2\pi]),$ and since $H^{1},$ $L^{2},$ and
$H^{-1}$ are all separable, reflexive spaces, we use the Aubin-Lions
Lemma to conclude that there exists a subsequence (which we do not
relabel) and a limit $u\in L^{2}(0,T;L^{2})$ such that
$u_{\varepsilon}$ converges to $u$ in this space.

Since $u_{\varepsilon}$ converges to $u$ in $L^{2}(0,T;L^{2}),$ we see
that for almost every $t,$ $u_{\varepsilon}(\cdot,t)$ converges to
$u(\cdot,t)$ in $L^{2}$ (along a further subsequence if necessary).
Furthermore, for every $t,$ we know that $u_{\varepsilon}(\cdot,t)$ is
bounded in $H^{1}$ (uniformly with respect to both $t$ and
$\varepsilon$).  Since the unit ball of a Hilbert space is weakly
compact, there exists a subsequence which converges to a weak limit in
$H^{1};$ this weak limit, however, must be equal to $u.$ So,
$u(\cdot,t)$ is, for almost every $t,$ in $H^{1},$ with a bound
independent of $t.$ This implies $u\in L^{\infty}([0,T];H^{1}).$ We
can then use Lemma \ref{interp} to conclude that $u_{\varepsilon}$
converges to $u$ in the space $L^{2}(0,T;H^{s'})$ for any
$s'\in[0,1).$

\

\noindent{\bf Step 4:} The limit solves the original equation.

The solutions, $\ueps$, of \eqref{uMollEquation} are classical
solutions, and thus are also solutions in the mild sense.  Therefore,
for all $\varepsilon>0,$ we can write
\begin{equation}
  \label{duhamelFormula}
  u_{\varepsilon}=e^{i\mathcal{J}_{\varepsilon}\partial_{x}^{2}t}u_{0}
  -\int_{0}^{t}e^{i\mathcal{J}_{\varepsilon}\partial_{x}^{2}(t-s)}
  \mathcal{J}_{\varepsilon}\left(|\mathcal{J}_{\varepsilon}u_{\varepsilon}|^{2\sigma}
    \mathcal{J}_{\varepsilon}u_{\varepsilon,x}\right)\ ds.
\end{equation}
We need to carefully take the limit of this to show that $u$ actually
solves \eqref{e:gdnls}, in the mild sense.

First, we apply an additional mollifier.  Fix $\delta>0$ and apply
$\mathcal{J}_{\delta}$ to \eqref{duhamelFormula}. Notice that
$J_{\delta}$ commutes with both the integral and the semigroup.  This
yields the following:
\begin{equation}\label{duhamelFormulaDelta}
  \mathcal{J}_{\delta}u_{\varepsilon}
  =e^{i\mathcal{J}_{\varepsilon}\partial_{x}^{2}t}\mathcal{J}_{\delta}u_{0}
  -\int_{0}^{t}e^{i\mathcal{J}_{\varepsilon}\partial_{x}^{2}(t-s)}\mathcal{J}_{\delta}
  \mathcal{J}_{\varepsilon}\left(|\mathcal{J}_{\varepsilon}u_{\varepsilon}|^{2\sigma}
    \mathcal{J}_{\varepsilon}u_{\varepsilon,x}\right)\ ds.\end{equation}
We will first take the limit of (\ref{duhamelFormulaDelta}) as $\varepsilon$ vanishes, and then
as $\delta$ vanishes.

If we consider $t\in[0,T],$ we may define the following operators:
$$B_{\varepsilon}:L^{2}\rightarrow L^{2},\qquad 
\mathcal{B}_{\varepsilon}:L^{2}(0,t;L^{2})\rightarrow L^{2},$$ with
\begin{align}
  B_{\varepsilon}f&\equiv e^{i\mathcal{J}_{\varepsilon}\partial_{x}^{2}t}f\\
  \mathcal{B}_{\varepsilon}f &\equiv
  \int_{0}^{t}B_{\varepsilon}(s-t)f\ ds
\end{align}
We similarly define the limiting operators $B$ and $\mathcal{B},$ as
follows:
$$B:L^{2}\rightarrow L^{2},\qquad 
\mathcal{B}:L^{2}(0,t;L^{2})\rightarrow L^{2},$$ with
\begin{align}
  Bf&\equiv e^{i\partial_{x}^{2}t}f,\\
  \mathcal{B}f &\equiv \int_{0}^{t}B(t-s)f\ ds.
\end{align}
These are all bounded linear operators between the given spaces, with
operator norms independent of $\varepsilon.$ We let $f_{\varepsilon}
=\mathcal{J}_{\varepsilon}\left(|\mathcal{J}_{\varepsilon}u_{\varepsilon}|^{2\sigma}
  \mathcal{J}_{\varepsilon}u_{\varepsilon,x}\right),$ and
$f=|u|^{2\sigma}u_{x}.$

To obtain the result of this step, we need to prove the convergence of
$B_{\varepsilon}\rightarrow B$ and
$\mathcal{B}_{\varepsilon}\rightarrow\mathcal{B}$ in their strong
operator topologies, and then show
$\mathcal{J}_{\delta}f_{\varepsilon}\rightarrow\mathcal{J}_{\delta}f$
in $L^{2}(0,t;L^{2})$.

\begin{lemma}
  \label{l:Bconv}
  Let $t\in[0, T]$. $B_{\varepsilon}\rightarrow B$ in the strong
  operator topology of $L^2\to L^2$, while
  $\mathcal{B}_{\varepsilon}\rightarrow\mathcal{B}$ in the strong
  operator toplogy of $L^2([0,t];L^2)\to L^2([0,t];L^2)$.
\end{lemma}

\begin{proof}
  We only prove the second, more complicated, convergence result.  Let
  $g\in L^{2}(0,t;L^{2})$ be given.  Using the triangle inequality and
  H\"{o}lder's inequality:
  \begin{multline}\nonumber
    \|(\mathcal{B}_{\varepsilon}-\mathcal{B})g\|_{L^{2}}
    =\left\|\int_{0}^{t}(B_{\varepsilon}(t-s)-B(t-s))g\ ds\right\|_{L^{2}}\\
    \leq \int_{0}^{t}\left\|(B_{\varepsilon}(t-s)-B(t-s))g
    \right\|_{L^{2}}\ ds\\
    \leq c\left(\int_{0}^{t}\left\|(B_{\varepsilon}(t-s)-B(t-s))g
      \right\|_{L^{2}}^{2}\ ds\right)^{1/2}.
  \end{multline}
  Using Plancherel's Theorem and Tonelli's Theorem, we then have:
  \begin{multline}
    \int_{0}^{t}\left\|(B_{\varepsilon}(t-s)-B(t-s))g
    \right\|_{L^{2}}^{2}\ ds\\
    =\int_{0}^{t}\sum_{k=-\infty}^{\infty}\left|\mathcal{F}\left((B_{\varepsilon}(t-s)-B(t-s))g\right)(k)\right|^{2}\
    ds\\=\sum_{k=-\infty}^{\infty}\int_{0}^{t}
    \left|\mathcal{F}\left((B_{\varepsilon}(t-s)-B(t-s))g\right)(k)\right|^{2}\ ds.\\
  \end{multline}
  Recall that the mollifier $\mathcal{J}_{\varepsilon}$ leaves low
  Fourier modes unchanged, while eliminating high modes, with the
  cutoff at $1/\varepsilon.$ Thus, all the low modes of
  $B_{\varepsilon}g$ and $Bg$ cancel, leaving only large wavenumbers.
  At large wavenumbers, $\mathcal{F}(B_{\varepsilon})(k)=1,$ since at
  those wavenumbers, we have
  $\mathcal{F}(\mathcal{J}_{\varepsilon})(k)=0.$ We are now able to
  see that the limit as $\varepsilon$ vanishes is zero:
  \begin{multline}
    \lim_{\varepsilon\rightarrow0^{+}}\|(\mathcal{B}_{\varepsilon}-\mathcal{B})g\|_{L^{2}}^2
    \leq \lim_{N\rightarrow\infty} \sum_{|k|\geq N}\int_{0}^{t}
    \left|\mathcal{F}\left((B_{\varepsilon}(t-s)-B(t-s))g\right)(k)\right|^{2}\ ds\\
    =\lim_{N\rightarrow\infty} \sum_{|k|\geq N}\int_{0}^{t}
    \left|\mathcal{F}[g-B(t-s)g](k)\right|^{2}\ ds=0.
  \end{multline}
  This is equal to zero because it is the limit of the tails of a
  convergent series.  We have shown that $\mathcal{B}_{\varepsilon}$
  converges to $\mathcal{B}$ in the strong operator topology.
\end{proof}

We next prove a result which will allow us to conclude
$\mathcal{J}_{\delta}f_{\varepsilon}\rightarrow\mathcal{J}_{\delta}f$
in $L^{2}(0,t;L^{2})$:
\begin{lemma}
  \label{convLemma}
  Let $t\in[0,T]$ and $\delta>0$ be given.  The sequence
  $\mathcal{J}_{\delta}f_{\varepsilon}$ converges to
  $\mathcal{J}_{\delta}f$ in $L^{2}(0,t;L^{2}).$
\end{lemma}

\begin{proof} This is essentially the same argument using Plancherel
  and Tonelli that was used above in the proof of Lemma \ref{l:Bconv},
  but with much more adding and subtracting:
  \begin{equation}
    \begin{split}
      &\mathcal{J}_{\delta}\mathcal{J}_{\varepsilon}(|\mathcal{J}_{\varepsilon}u_{\varepsilon}|^{2\sigma}
      \mathcal{J}_{\varepsilon}u_{\varepsilon,x})
      -\mathcal{J}_{\delta}(|u|^{2\sigma}u_{x})\\
      &=\underbrace{\mathcal{J}_{\delta}\Bigg[
        \mathcal{J}_{\varepsilon}(|\mathcal{J}_{\varepsilon}u_{\varepsilon}|^{2\sigma}
        \mathcal{J}_{\varepsilon}u_{\varepsilon,x})
        -\mathcal{J}_{\varepsilon}(|\mathcal{J}_{\varepsilon}u_{\varepsilon}|^{2\sigma}
        \mathcal{J}_{\varepsilon}u_{x})\Bigg]}_{A_1}
      \\
      \quad&+\underbrace{\mathcal{J}_{\delta}\Bigg[\mathcal{J}_{\varepsilon}(|\mathcal{J}_{\varepsilon}u_{\varepsilon}|^{2\sigma}
        \mathcal{J}_{\varepsilon}u_{x})
        -\mathcal{J}_{\varepsilon}(|\mathcal{J}_{\varepsilon}u_{\varepsilon}|^{2\sigma}u_{x})\Bigg]}_{A_2}
      \\
      \quad &+\underbrace{\mathcal{J}_{\delta}
        \Bigg[\mathcal{J}_{\varepsilon}(|\mathcal{J}_{\varepsilon}u_{\varepsilon}|^{2\sigma}u_{x})
        -\mathcal{J}_{\varepsilon}(|\mathcal{J}_{\varepsilon}u|^{2\sigma}u_{x})\Bigg]}_{A_3}
      \\
      \quad
      &+\underbrace{\mathcal{J}_{\delta}\Bigg[\mathcal{J}_{\varepsilon}(|\mathcal{J}_{\varepsilon}u|^{2\sigma}u_{x})
        -\mathcal{J}_{\varepsilon}(|u|^{2\sigma}u_{x})\Bigg]}_{A_4}
      +\underbrace{\mathcal{J}_{\delta}\Bigg[\mathcal{J}_{\varepsilon}(|u|^{2\sigma}u_{x})-|u|^{2\sigma}u_{x}\Bigg]}_{A_5}
      \\
      &=A_{1}+A_{2}+A_{3}+A_{4}+A_{5}.
    \end{split}
  \end{equation}
  Of these, $A_2$, $A_4$ and $A_5$ go to zero as $\varepsilon\to 0$ by
  Lemma \ref{plancherelTonelli} since they involve differences between
  a mollified and an unmollified quantity.  The other terms, $A_1$ and
  $A_3,$ go to zero because $u_{\varepsilon}$ converges to $u$.  To
  demonstrate these convergences in detail, we will rely on the
  following results from Steps 2 and 3:
  \begin{itemize}
  \item The $\ueps$ solutions are uniformly bounded in $L^\infty(0,T;
    H^1)$;
  \item $u \in L^\infty(0,T; H^1)$;
  \item $\ueps \to u$ in $L^2(0,T;H^{s'})$ for all $0\leq s'<1$.
  \end{itemize}

  For $A_{1},$ we must take care, since we do not know that
  $u_{\varepsilon}$ converges to $u$ in $L^{2}(0,t;H^{1}).$
  $\mathcal{J}_{\delta}$ was introduced just to deal with this
  difficulty.  We first take the derivative away from
  $u_{\varepsilon,x}-u_{x},$ using the product rule:
  \begin{equation*}
    \begin{split}
      A_1=\mathcal{J}_{\delta}\mathcal{J}_{\varepsilon}\left(|\mathcal{J}_{\varepsilon}u_{\varepsilon}|^{2\sigma}\partial_{x}(u_{\varepsilon}-u)\right)
      &=\underbrace{\mathcal{J}_{\delta}\mathcal{J}_{\varepsilon}\partial_{x}\Big(|\mathcal{J}_{\varepsilon}u_{\varepsilon}|^{2\sigma}(u_{\varepsilon}-u)\Big)}_{A_{11}}\\
      &\quad
      \underbrace{-\mathcal{J}_{\delta}\mathcal{J}_{\varepsilon}\left((u_{\varepsilon}-u)\partial_{x}\left(|\mathcal{J}_{\varepsilon}u_{\varepsilon}|^{2\sigma}\right)\right)}_{A_{12}}.
    \end{split}
  \end{equation*}
  For $A_{11},$ we use the fact that
  $\mathcal{J}_{\delta}\partial_{x}$ is a bounded operator, with the
  bound on the operator norm being clearly independent of
  $\varepsilon.$ Then we have the following:
  \begin{equation}\nonumber
    \int_{0}^{t}\|A_{11}\|_{L^{2}}^{2}\ ds \leq c\int_{0}^{t}\int_{0}^{2\pi}
    |\mathcal{J}_{\varepsilon}u_{\varepsilon}|^{4\sigma}|u_{\varepsilon}-u|^{2}\ dxds
    \leq c\int_{0}^{t}\int_{0}^{2\pi}|u_{\varepsilon}-u|^{2}\ dxds,
  \end{equation}
  where we have used the uniform boundedness of the $\ueps$ in
  $L^\infty(0,T;H^1)$.  Since $\ueps\to u$ in $L^2(0,T; L^2)$,
  $A_{11}\to 0$ in $L^2(0,t;L^2)$.

  For $A_{12}$,
  \begin{equation*}
    \begin{split}
      \int_0^t \norm{A_{12}}_{L^2}^2 ds &\leq \int_0^t \norm{\ueps
        -u}_{L^\infty}^2 \norm{\partial_x(\abs{\moll
          \ueps}^{2\sigma})}_{L^2}^2 ds\\
      &\leq \norm{\partial_x(\abs{\moll \ueps}^{2\sigma})}_{L^\infty_t
        L^2_x}^2 \int_0^t \norm{\ueps
        -u}_{L^\infty}^2 ds\\
      &\leq c \int_0^t \norm{\ueps -u}_{H^{3/4}}^2 ds
    \end{split}
  \end{equation*}
  where we have again used the uniform boundedness of the $\ueps$ in
  $L^\infty(0,T;H^1),$ as well as Sobolev embedding.  Since $\ueps \to
  u$ in $L^2(0,T; H^{3/4})$, we have our result for $A_{12}$, and
  $A_1$ is done.

  We treat $A_{3}$ similarly:
  \begin{equation}\nonumber
    \int_{0}^{t}\|A_{3}\|_{L^{2}}^{2}\ ds \leq \int_{0}^{t}\int_{0}^{2\pi} |u_{x}|^{2}
    \Bigg| |\mathcal{J}_{\varepsilon}u_{\varepsilon}|^{2\sigma}-|\mathcal{J}_{\varepsilon}u|^{2\sigma}
    \Bigg|^{2}\ dxds.
  \end{equation}
  Since $|z|^{2\sigma}$ is Lipschitz continuous, and using our uniform
  bounds, we can estimate this as
  \begin{equation}\nonumber
    \int_{0}^{t}\|A_{3}\|_{L^{2}}^{2}\ ds \leq c
    \int_{0}^{t}\|u_{\varepsilon}-u\|_{L^{\infty}}^{2}\ ds\\
    \leq c \int_{0}^{t}\|u_{\varepsilon}-u\|_{H^{3/4}}^{2}\ ds
  \end{equation}
  This vanishes for the same reason as above.  Hence $A_3\to 0$.

  We will investigate $A_{5},$ but we will omit the details of the
  estimates for $A_{2}$ and $A_{4},$ as they are essentially the same
  as the estimate of $A_{5}.$ We begin by writing the following:
  \begin{equation}\label{estimateA5}
    \int_{0}^{t}\|A_{5}\|_{L^{2}}^{2}\ ds 
    \leq\int_{0}^{t} 
    \|\mathcal{J}_{\varepsilon}(|u|^{2\sigma}u_{x})-|u|^{2\sigma}u_{x}\|_{L^{2}}^{2}\ ds.
  \end{equation}
  Since $|u|^{2\sigma}u_{x}$ is an element of $L^{2}(0,T;L^{2}),$ we
  see from Lemma \ref{plancherelTonelli} that the right-hand side of
  (\ref{estimateA5}) goes to zero as $\varepsilon$ vanishes.

  The same argument applies to show that $A_{2}$ and $A_{4}$ converge
  to zero.  We have now established that
  $\mathcal{J}_{\delta}f_{\varepsilon}\rightarrow\mathcal{J}_{\delta}f$
  in $L^{2}(0,t;L^{2}),$ for any $t\in[0,T].$
\end{proof}

We have now shown that for each time $t\in[0,T],$ the following holds:
\begin{equation}\label{duhamelHoldsDeltaLimit}
  \mathcal{J}_{\delta}u_{\varepsilon}\rightarrow 
  e^{i\partial_{x}^{2}t}\mathcal{J}_{\delta}u_{0}-\int_{0}^{t}e^{i\partial_{x}^{2}(s-t)}
  \mathcal{J}_{\delta}(|u|^{2\sigma}u_{x})\ ds,
\end{equation}
with the convergence being in $L^{2}.$ Since $u_{\varepsilon}$
converges to $u$ at almost every time (along a subsequence), we
conclude that
\begin{equation}\label{duhamelHoldsDelta}
  \mathcal{J}_{\delta}u= e^{i\partial_{x}^{2}t}\mathcal{J}_{\delta}u_{0}
  -\int_{0}^{t}e^{i\partial_{x}^{2}(s-t)}\mathcal{J}_{\delta}(|u|^{2\sigma}u_{x})\ ds,
\end{equation}
for almost every $t\in[0,T].$ We now must take the limit as $\delta$
vanishes.

Since $|u|^{2\sigma}u_{x}$ is in $L^{2}(0,T;L^{2}),$ by Lemma
\ref{plancherelTonelli}, we are able to take the limit in the integral
in (\ref{duhamelHoldsDelta}) as $\delta$ vanishes.  By
(\ref{mollifierLimit}), we are able to take the limit in the other
terms in (\ref{duhamelHoldsDelta}) as $\delta$ vanishes.  Thus, we
find the following:
\begin{equation}\label{duhamelHolds}
  u= e^{i\partial_{x}^{2}t}u_{0}
  -\int_{0}^{t}e^{i\partial_{x}^{2}(t-s)}|u|^{2\sigma}u_{x}\ ds,
\end{equation}
for almost all $t\in[0,T]$, with the above equality holding in the
sense of $L^2$.  We mention that for our $u\in L^{\infty}(0,T;H^{1}),$
the right-hand side of (\ref{duhamelHolds}) makes sense for all
$t\in[0,T],$ rather than just almost every $t,$ and the value of this
right-hand side would not change if we altered the definition of $u$
on a set of times of measure zero.  Therefore, we define $u$ to be
equal to the right-hand side of (\ref{duhamelHolds}) for the times at
which (\ref{duhamelHolds}) did not already hold, which is indeed a set
of measure zero.  Thus, we may say that (\ref{duhamelHolds}) is true
for all $t\in[0,T],$ and furthermore, $\mathcal{J}_{\delta}u$
converges to $u$ in $L^{2}$ for all $t\in[0,T].$

We mention now that while we have shown $u\in L^{\infty}(0,T; H^{1}),$
it is actually the case that for all $t\in[0,T],$ we have
$u(\cdot,t)\in H^{1}.$ To see this, we begin by fixing $t\in[0,T].$ We
have established just above that
$\mathcal{J}_{\delta}u_{\varepsilon}(\cdot,t)$ converges in $L^{2}$ to
$\mathcal{J}_{\delta}u(\cdot,t).$ We know, however, that
$\mathcal{J}_{\delta}u_{\varepsilon}(\cdot,t)$ is bounded in $H^{1},$
uniformly in $\delta$ and $\varepsilon;$ thus, there exists a
subsequence which converges weakly in $H^{1}$ to some limit, as
$\varepsilon$ vanishes.  This limit, however, must be
$\mathcal{J}_{\delta}u(\cdot,t);$ thus,
$\mathcal{J}_{\delta}u(\cdot,t)$ is not only in $H^{1},$ but satisfies
the same bound (which is uniform with respect to $\delta$).  We have
also shown just above that $\mathcal{J}_{\delta}u(\cdot,t)$ converges
to $u(\cdot,t)$ in $L^{2}.$ Again, since
$\mathcal{J}_{\delta}u(\cdot,t)$ is uniformly bounded in $H^{1},$ it
has a weak limit in $H^{1}$ along a subsequence.  This limit must,
however, be equal to $u(\cdot,t),$ and thus we see that $u(\cdot,t)$
is in $H^{1}$ for every $t\in[0,T].$ Furthermore, we have $u(\cdot,t)$
bounded in $H^{1},$ uniformly with respect to $t.$

We have almost established all of the claims of Theorem
\ref{mainTheorem}.  All that remains is to demonstrate continuity in
time, below the highest spatial regularity.  This is the content of
the following lemma.

\begin{lemma}
  \label{l:h1cont_in_time}
  Assume $u\in L^\infty(0,T;H^1)$ is a mild solution, then $u\in
  C(0,T;H^{s'})$ for all $0\leq s'<1$.
\end{lemma}
\begin{proof}

  Taking the difference, in $L^2$,
  \begin{equation*}
    \begin{split}
      \norm{u(t+h) - u(t)}_{L^2}& \leq
      \norm{ \paren{e^{i\partial_{x}^{2}(t+h)}-
          e^{i\partial_{x}^{2}t}}u_0}_{L^2}\\
      &\quad + \norm{\int_0^{t+h} e^{i\partial_{x}^{2}(t+
          h-s)}\abs{u}^{2\sigma}u_x -\int_0^{t} e^{i\partial_{x}^{2}(t
          -
          s)}\abs{u}^{2\sigma}u_x    }_{L^2}\\
      &\leq \norm{e^{i\partial_{x}^{2}h}-I}_{L^2\to
        L^2}\norm{u_0}_{L^2}
      +\int_{t}^{t+h} \norm{\abs{u}^{2\sigma} u_x}_{L^2} ds\\
      &\quad + \norm{e^{i\partial_{x}^{2}h}-I}_{L^2\to L^2}\int_0^t
      \norm{\abs{u}^{2\sigma} u_x}_{L^2} ds\\
      &\leq \norm{e^{i\partial_{x}^{2}h}-I}_{L^2\to
        L^2}\paren{\norm{u_0}_{L^2} + T \norm{u}_{L^\infty_t
          H^1_x}^{2\sigma+1}}\\
      &\quad + h \norm{u}_{L^\infty_t H^1_x}^{2\sigma+1}
    \end{split}
  \end{equation*}
  By the continuity of the Schr\"odinger semigroup, the right-hand
  side vanishes as $h\to 0$, so $u \in C(0, T;L^2)\cap
  L^\infty(0,T;H^1)$. Using the ineterpolation estimate of Lemma
  \ref{interp}, we actually obtain $u \in C(0, T;H^{s'})\cap
  L^\infty(0,T;H^1)$ for all $0\leq s' < 1$.
\end{proof}

The proof of Theorem \ref{mainTheorem} is complete.  We are unable to
conclude that the solutions we have constructed are continuous in time
with values in $H^{1};$ at present, the most we can prove is that the
solution is weakly continuous in time with values in $H^{1}.$ We
establish this and other additional properties of the $H^{1}$
solutions in the next subsection.

\subsection{Further properties of $H^{1}$ solutions}

The additional properties that we establish are motivated towards
showing that small data is global in time in $H^1$.  We are not able
to prove this statement about global existence, but we will discuss in
Section \ref{towardsGlobal} below how much more is needed to close the
gap.

While we have established above that the solution $u$ is in $H^{1}$ at
each time, we have only established that $u$ is continuous in $H^{s}$
for $s<1.$ We are not able to establish continuity in time in $H^{1};$
the best we can obtain is weak continuity in time, which is the
content of the following lemma.  We note that to establish continuity
in time, given that we do establish weak continuity, it would only be
necessary to prove that the $H^{1}$ norm of $u$ is continuous in time.
We note that in Section \ref{H2Solutions} below, for solutions with
initial data in $H^{2},$ we are able to establish continuity of the
highest norm, and thus continuity in $H^{2}.$

\begin{theorem} \label{weakTheoremContTime} Let $k>0,$ and let $u\in
  L^{\infty}(0,T;H^{k})\cap C(0,T;H^{s}),$ for all $s\in [0,k).$ Then
  $u$ is weakly continuous in time with values in $H^{k}.$
\end{theorem}
\begin{proof}
  Let $t\in[0,T].$ We will establish the following convergence, for
  any $\phi\in H^{k}:$
  \begin{equation}\label{weakConvH1}
    \lim_{s\rightarrow t}\langle u(\cdot,s), \phi\rangle_{H^{k}} = \langle u(\cdot,t),\phi\rangle_{H^{k}}.
  \end{equation}
  (The limit is of course taken to be a one-sided limit if either
  $t=0$ or $t=T.$) We recall that the dual of $H^{k}$ can be viewed
  either as being equal to $H^{k}$ itself, or alternatively, as being
  equal to $H^{-k}.$ Thus, given $\phi\in H^{k},$ if we define the
  bounded linear functional $L_{\phi}$ by $$L_{\phi}(f)=\langle
  f,\phi\rangle_{H^{k}}, \qquad \forall f\in H^{k},$$ then there
  exists $\tilde{\phi}\in H^{-k}$ such that for all $f\in H^{k},$ we
  have $L_{\phi}(f)=\langle f,\tilde{\phi}\rangle_{L^{2}}.$ Therefore,
  in order to establish (\ref{weakConv}), it is sufficient to show
  that for all $\tilde{\phi}\in H^{-k},$ we have
  \begin{equation}\label{hMinus1}
    \lim_{s\rightarrow t}\langle u(\cdot,s),\tilde{\phi}\rangle_{L^{2}}
    =\langle u(\cdot,t),\tilde{\phi}\rangle_{L^{2}}.
  \end{equation}

  Let $\epsilon>0$ be given.  Let $K>0$ be such that for all
  $s\in[0,T],$ we have $\|u\|_{H^{k}}\leq K.$ Let $\psi\in H^{-k}$ be
  given.  Since $H^{-k+1/2}$ is dense in $H^{-k},$ we can find
  $\psi_{\epsilon}\in H^{-k+1/2}$ such that
  $\|\psi_{\epsilon}-\psi\|_{H^{-k}}<\frac{\epsilon}{3K}.$ Then, we
  add and subtract and use the triangle inequality as follows:
  \begin{multline}\label{epsilonOver3}
    \left|\langle u(\cdot,s),\psi\rangle_{L^{2}} - \langle
      u(\cdot,t),\psi\rangle_{L^{2}}\right| \leq \left|\langle
      u(\cdot,s),\psi-\psi_{\epsilon}\rangle_{L^{2}}\right|
    \\
    +\left|\langle
      u(\cdot,s)-u(\cdot,t),\psi_{\epsilon}\rangle_{L^{2}}\right|
    +\left|\langle
      u(\cdot,t),\psi_{\epsilon}-\psi\rangle_{L^{2}}\right|.
  \end{multline}
  For the first of these, we have
$$\left|\langle u(\cdot,s),\psi-\psi_{\epsilon}\rangle_{L^{2}}\right| 
\leq \|u(\cdot,s)\|_{H^{k}}\|\psi-\psi_{\epsilon}\|_{H^{-k}}\leq
\frac{\epsilon}{3}.$$ Similarly, the third term on the right-hand side
of (\ref{epsilonOver3}) is also at most $\frac{\epsilon}{3}.$ For the
second term on the right-hand side of (\ref{epsilonOver3}), we
estimate it as
$$\left|\langle u(\cdot,s)-u(\cdot,t),\psi_{\epsilon}\rangle_{L^{2}}\right|
\leq
\|u(\cdot,s)-u(\cdot,t)\|_{H^{k-1/2}}\|\psi_{\epsilon}\|_{H^{-k+1/2}}.$$
Since $u\in C([0,T];H^{k-1/2}),$ there exists $\eta>0$ such that if
$|s-t|<\eta,$ then
$$\|u(\cdot,s)-u(\cdot,t)\|_{H^{k-1/2}}\leq \frac{\epsilon}{3(1+\|\psi_{\epsilon}\|_{H^{-k+1/2}})}.$$

We conclude that for all $s$ such that $|s-t|<\eta,$ we have
$$\left|\langle u(\cdot,s)-u(\cdot,t),\psi\rangle_{L^{2}}\right|<\epsilon.$$
This implies (\ref{hMinus1}), and as we have argued, this implies
(\ref{weakConvH1}).
\end{proof}

The other properties of our $H^{1}$ solutions which we note are
related to the conserved quantities discussed in the introduction.  It
is helpful to note that the mollified equations also have the
following conservation property, for which we omit the proof:
\begin{lemma}
  \label{l:massconservation}
  The mollified evolution conserves the mass and momentum invariants,
  \eqref{e:mass} and \eqref{e:momentum}, as stated.
\end{lemma}
We are able to conclude that the limit, $u,$ also conserves mass and
momentum.  This is the content of the following lemma.
\begin{lemma}
  \label{l:massmomentum}
  The solution that has been constructed, $u \in L^\infty(0,T;H^1)\cap
  C(0,T;H^{s'})$ with $0\leq s' < 1$, conserves the mass and momentum
  invariants, \eqref{e:mass} and \eqref{e:momentum}.
\end{lemma}
\begin{proof}
  Recall that, up to subsequence extraction, $\ueps\to u$ in
  $L^2([0,T]; L^2)$ and, for almost all $t$, $\ueps(\cdot,t)
  \rightharpoonup u(\cdot, t)$ in $H^1$.  For the mass invariant,
  \begin{equation*}
    \begin{split}
      \abs{\mathcal{M}[u(t)] - \mathcal{M}[u_0]} &\leq
      \abs{\mathcal{M}[u(t)]
        - \mathcal{M}[\ueps(t)]} \\
      &\quad + \abs{\mathcal{M}[\ueps(t)] - \mathcal{M}[\ueps(0)]}+
      \abs{\mathcal{M}[\ueps(0)] - \mathcal{M}[u_0]}
    \end{split}
  \end{equation*}
  Since the data is not regularized, the third term vanishes.  The
  second term vanishes because the regularized flow conserves mass.
  Finally, the first term vanishes as $\varepsilon\to 0$ since we have
  convergence in $L^2$ for almost all $t$.

  For the momentum invariant,
  \begin{equation*}
    \begin{split}
      \abs{\mathcal{P}[u(t)] - \mathcal{P}[u_0]} &\leq
      \abs{\mathcal{P}[u(t)]
        - \mathcal{P}[\ueps(t)]} \\
      &\quad + \abs{\mathcal{P}[\ueps(t)] - \mathcal{P}[\ueps(0)]}+
      \abs{\mathcal{P}[\ueps(0)] - \mathcal{P}[u_0]}.
    \end{split}
  \end{equation*}
  As before, the second two terms vanish exactly.  We are thus left to
  consider the first term,
  \begin{equation*}
    \begin{split}
      \int \bar u u_x - \int \bar u_{\varepsilon} u_{\varepsilon,x} &=
      \int \bar u (u_x - u_{\varepsilon,x}) + \int (\bar u -
      \bar{u}_{\varepsilon})u_{\varepsilon,x}.
    \end{split}
  \end{equation*}
  Since $u_{\varepsilon} \rightharpoonup u$ in $H^1$ for almost all
  $t$, the first integral vanishes as $\varepsilon \to 0$.  Since we
  have convergence in $L^2$ for almost all $t$, and uniform
  boundedness of the $u_{\varepsilon}$ sequence in $H^1$, the second
  integral also vanishes.
\end{proof}

We will see below in Section 5 that the Hamiltonian is conserved for
smoother solutions than we are considering at present.  For solutions
in the energy space, the solutions we have constructed satisfy the
following result:
\begin{lemma}
  \label{l:hamiltonian2}
  For the constructed solution, $u \in L^\infty(0,T;H^1)\cap
  C(0,T;H^{s'})$ with $0\leq s' < 1$,
  \begin{equation*}
    \mathcal{H}[u(t)] \leq \mathcal{H}[u_0]
  \end{equation*}

\end{lemma}

\begin{proof}
  We begin by splitting the Hamiltonian functional into kinetic and
  potential parts as $\mathcal{H}[u] = \mathcal{K}[u] +
  \mathcal{V}[u]$, with
  \begin{equation*}
    \mathcal{K}[u] = \norm{u_x}_{L^2}^2, \quad \mathcal{V}[u] = \frac{1}{(\sigma+1)^2}\int \bar u ^{\sigma+1} D_x u^{\sigma+1}.
  \end{equation*}
  The regularized Hamlitonian is analogously split
  $\mathcal{H}_\varepsilon[u] = \mathcal{K}[u] +
  \mathcal{V}_{\varepsilon}[u]$:
  \begin{equation*}
    \mathcal{K}[u] = \norm{u_x}_{L^2}^2, \quad \mathcal{V}_\varepsilon[u] = \frac{1}{(\sigma+1)^2}\int (\mathcal{J}_{\varepsilon}\bar u) ^{\sigma+1} D_x (\mathcal{J}_{\varepsilon}u)^{\sigma+1}.
  \end{equation*}

  We first show that for all $t\in[0,T]$,
  \begin{equation*}
    \lim_{\varepsilon\to 0} \mathcal{V}_\varepsilon[u_\varepsilon(t)] = \mathcal{V}[u(t)] 
  \end{equation*}
  Suppressing $t$ and letting $v_\varepsilon \equiv
  \mathcal{J}_{\varepsilon} u_{\varepsilon}$, this follows by the
  direct calculation:
  \begin{equation*}
    \begin{split}
      \abs{\mathcal{V}_\varepsilon[u_\varepsilon] - \mathcal{V}[u]  }&\leq \frac{1}{(\sigma+1)^2}\int  \abs{(\bar v_{\varepsilon}^{\sigma+1} - \bar u^{\sigma+1})\partial_x v_{\varepsilon}^{\sigma+1}} \\
      &\quad+ \frac{1}{(\sigma+1)^2}\int \abs{ (v_{\varepsilon}^{\sigma+1} - u^{\sigma+1})\partial_x\bar u^{\sigma+1}}\\
      &  \leq  \frac{1}{\sigma+1} \paren{\norm{u^{\sigma}u_x}_{L^2}+\norm{v_{\varepsilon}^{\sigma}v_{\varepsilon,x}}_{L^2}}\norm{v_{\varepsilon}^{\sigma+1} - u^{\sigma+1}}_{L^2}\\
      &
      \leq \paren{\norm{u}_{H^1}^{\sigma+1}+\norm{v_{\varepsilon}}_{H^1}^{\sigma+1}}\max\{\norm{u}_{L^\infty},\norm{v_{\varepsilon}}_{L^\infty}\}^{\sigma}
      \norm{v_{\varepsilon} - u}_{L^2}
    \end{split}
  \end{equation*}
  Recall that $\norm{v_\varepsilon}_{H^1} \leq
  \norm{u_\varepsilon}_{H^1}$, and $u_\varepsilon$ are uniformly
  bounded in $H^1$ with respect to $t$ and $\varepsilon$.
  Furthermore, the limit, $u$, is in $L^\infty_t H^1_x$.  Thus, all
  that needs to be checked is that $\norm{v_{\varepsilon} - u}_{L^2}$
  vanishes as $\varepsilon \to 0$.  This is immediate since
  $u_\varepsilon \to u$ in $L^2_tL^2_x$ and
  \begin{equation*}
    \begin{split}
      \norm{ \mathcal{J}_{\varepsilon} u_{\varepsilon}- u}_{L^2}&\leq \norm{\mathcal{J}_{\varepsilon} (u_{\varepsilon}- u)}_{L^2} + \norm{( \mathcal{J}_{\varepsilon} -I) u}_{L^2}\\
      & \leq \norm{u_{\varepsilon}- u}_{L^2} + \norm{(
        \mathcal{J}_{\varepsilon} -I) u}_{L^2}.
    \end{split}
  \end{equation*}

  To proceed, we remark that $\mathcal{H}_\varepsilon$ is continuous
  with respect to $\varepsilon$, so that
  \begin{equation*}
    \lim_{\varepsilon\to 0}\mathcal{H}_\varepsilon[u_0] = \mathcal{H}[u_0].
  \end{equation*}
  Therefore, given $u_0$ and $t$, and $\delta>0$, for all sufficiently
  small $\varepsilon$,
  \begin{equation*}
    \abs{\mathcal{H}_\varepsilon[u_0] - \mathcal{H}[u_0]}\leq \frac{\delta}{3}, \quad\abs{\mathcal{V}_\varepsilon[u_\varepsilon(t)] - \mathcal{V}[u(t)]}\leq \frac{\delta}{3}.
  \end{equation*}
  Additionally, since $u_\varepsilon(t) \rightharpoonup u(t)$ in $H^1$
  for almost all $t$,
  \begin{equation*}
    \mathcal{K}[u(t)]  = \norm{u_x(t)}_{L^2}^2 \leq \liminf_{\varepsilon\to 0} \norm{u_{\varepsilon,x}(t)}_{L^2}^2\leq \mathcal{K}[u_{\varepsilon}(t)]  +\frac{\delta}{3}
  \end{equation*}
  provided $\varepsilon$ is sufficiently small.  This implies the
  following:
  \begin{equation*}
    \begin{split}
      \mathcal{H}[u(t)] = \mathcal{K}[u(t)] +  \mathcal{V}[u(t)] &\leq  \mathcal{K}[u_{\varepsilon}(t)] +  \mathcal{V}_{\varepsilon}[u_{\varepsilon}(t)]  + \frac{2\delta }{3}\\
      & \leq \mathcal{H}_{\varepsilon}[u_0] + \frac{2\delta }{3}\leq
      \mathcal{H}[u_0] + \delta.
    \end{split}
  \end{equation*}
  Since $\delta>0$ was arbitrary, we are done.

\end{proof}

\

\section{Well-posedness in $H^{2}$}\label{H2Solutions}

If we are willing to work with solutions in $H^2$, stronger results
can be obtained.  Specifically, in this section, we prove Theorem
\ref{t:h2}.

\noindent {\bf Step 1:} Short-time existence.

Assume now that $u_{0}\in H^{2}.$ Using the same mollified equation,
\eqref{uMollEquation}, we can immediately obtain local in time
solutions for all $\vareps>0$, in $C^1(-T_\vareps, T_\vareps, H^2)$.
The only difference is that in obtaining the Lipschitz continuity of
the right-hand side, we use (\ref{mollifierGains}) with $s=2:$
\begin{equation*}
  \|\mathcal{J}_{\varepsilon}F_{\varepsilon}(f)-\mathcal{J}_{\varepsilon}F_{\varepsilon}(g)\|_{H^{2}}\leq \frac{c}{\epsilon^2}\| F_{\varepsilon}(f)-F_{\varepsilon}(g)\|_{L^{2}}
\end{equation*}

\noindent {\bf Step 2:} Existence on a uniform time interval.  For
solutions in $H^2$, it is sufficient to work with $\|\ueps\|_{H^2}^2$
as the energy.  Indeed,
\begin{lemma}
  For the solutions $\ueps\in C^1(-T_\vareps, T_\vareps, H^2)$,
  \begin{equation}
    \frac{d}{dt}\|\ueps\|_{H^2}^2 \leq  (6\sigma + 8 \sigma^2)\|\ueps\|_{H^2}^{4\sigma + 4}
  \end{equation}
\end{lemma}
\begin{proof}
  This can be obtained by a direct calculation.  Note that the time
  derivative of $\int |\ueps|^2$ vanishes because it is conserved.
  \begin{equation*}
    \begin{split}
      \frac{d}{dt}\|\ueps\|_{H^2}^2 &= \int -\moll \bar u_{\vareps,xxxx} \abs{\moll \ueps}^{2\sigma} \moll u_{\vareps,x} + \cc\\
      &  = \int \moll \bar u_{\vareps,xxx} \left(\abs{\moll \ueps}^{2\sigma} \moll u_{\vareps,x}\right)_x + \cc\\
      & = \int \left(\abs{ \moll u_{\vareps,xx}}^2\right)_x\abs{\moll \ueps}^{2\sigma}\\
      & \quad + \int \moll \bar u_{\vareps,xxx} \moll u_{\vareps,x} \left(\abs{\moll \ueps}^{2\sigma} \right)_x + \cc\\
      & = -3\underbrace{\int\abs{ \moll
          u_{\vareps,xx}}^2\paren{\abs{\moll
            \ueps}^{2\sigma}}_x}_{I_1}
      -\underbrace{\int\paren{\abs{\moll
            u_{\vareps,x}}^2}_x \paren{\abs{\moll
            \ueps}^{2\sigma}}_{xx}}_{I_2}
    \end{split}
  \end{equation*}
  For the first term,
  \begin{equation*}
    \begin{split}
      I_1 & = \int \abs{ \moll u_{\vareps,xx}}^2\paren{ \sigma\abs{\moll \ueps}^{2\sigma-2} \moll \bar \ueps \moll u_{\vareps,x} + \cc}\\
      & \leq  2\sigma \norm{\moll \ueps}_{L^\infty}^{2\sigma-1} \norm{\moll u_{\vareps,x}}_{L^\infty} \norm{\moll u_{\vareps,xx}}_{L^2}^2\\
      & \leq 2\sigma \norm{\ueps}_{H^1}^{2\sigma-1}
      \norm{\ueps}_{H^2}^3\leq 2 \sigma \norm{\ueps}_{H^2}^{2\sigma
        +2}
    \end{split}
  \end{equation*}
  For the other term,
  \begin{equation*}
    \begin{split}
      I_2 & =\int \paren{\moll \bar u_{\vareps,x}\moll u_{\vareps,xx} +\cc }\paren{\sigma  \moll \bar \ueps^\sigma \moll \ueps^{\sigma-1} \moll u_{\vareps,x} + \cc  }_x\\
      & = \int \paren{\moll \bar u_{\vareps,x}\moll u_{\vareps,xx} +\cc }\\
      &\quad \quad \times \paren{\sigma^2  \abs{\moll \bar \ueps}^{2\sigma-2} \abs{ \moll u_{\vareps,x}}^2 +\sigma(\sigma-1) \moll \bar \ueps^\sigma \moll \ueps^{\sigma-2} \moll u_{\vareps,x}^2 + \sigma  \moll \bar \ueps^\sigma \moll \ueps^{\sigma-1} \moll u_{\vareps,xx}  + \cc  }\\
      &\leq  4 \sigma^2 \norm{\moll  \ueps}_{L^\infty}^{2\sigma-2} \norm{\moll u_{\vareps,x}}_{L^\infty}^2\norm{\moll u_{\vareps,x}}_{L^2}\norm{\moll u_{\vareps,xx}}_{L^2}\\
      &\quad +4 \sigma(\sigma-1)  \norm{\moll \ueps}_{L^\infty}^{2\sigma-2} \norm{\moll u_{\vareps,x}}_{L^\infty}^2\norm{\moll u_{\vareps,x}}_{L^2}\norm{\moll u_{\vareps,xx}}_{L^2}\\
      &\quad + 4\sigma \norm{\moll \bar \ueps}_{L^\infty}^{2\sigma-1}  \norm{\moll u_{\vareps,x}}_{L^\infty}\norm{\moll u_{\vareps,xx}}_{L^2}^2\\
      &\leq (8\sigma^2-4\sigma) \norm{\ueps}_{H^1}^{2\sigma-1} \norm{\ueps}_{H^2}^3 + 4\sigma  \norm{\ueps}_{H^1}^{2\sigma-1} \norm{\ueps}_{H^2}^3 \\
      &\leq 8\sigma^2 \norm{\ueps}_{H^2}^{2\sigma+2}
    \end{split}
  \end{equation*}
  Consequently,
  \begin{equation*}
    \frac{d}{dt}\|\ueps\|_{H^2}^2 \leq \paren{6\sigma + 8\sigma^2}\norm{\ueps}_{H^2}^{2\sigma+2}.
  \end{equation*}
\end{proof}
We are now assured that these solutions cannot blow up arbitrarily
fast so there exists a $T$, independent of $\vareps$, for which the
solutions must exist.

\noindent{\bf Step 3:} Passage to the Limit

As before, we are able to pass the the limit by using the Aubin-Lions
Lemma.  We find $u_{\varepsilon}$ converges (along a subsequence) to
$u\in L^{2}(0,T; L^{2}).$ In the same way as before, we are able to
further conclude that $u\in L^{\infty}(0,T; H^{2}),$ and
$u_{\varepsilon}$ converges to $u$ in $L^{2}(0,T; H^{s'}),$ for any
$s'\in [0,2).$ Following the same arguments as before, we are able to
conclude that $u$ is a mild solution of (\ref{e:gdnls}), i.e., $u$
satisfies (\ref{duhamelHolds}).  Lemma \ref {l:h1cont_in_time} also
generalizes, allowing us to infer $u \in C(0,T;H^{s})$ for all $0\leq
s<2$.  In fact, we will show below that the solution is actually
continuous in time in $H^2$.

In the present setting, we can obtain differentiability in time:
\begin{lemma}
  \label{l:c1}
  Assume $u\in L^\infty(0,T;H^2)\cap C(0,T;H^{1})$ is a mild solution.
  Then $u_t\in L^{\infty}(0,T;L^2)$ and $u_t= i u_{xx} -
  \abs{u}^{2\sigma} u_x$.
\end{lemma}
\begin{proof}
  We begin by writing
  \begin{equation*}
    \begin{split}
      &\norm{h^{-1}\paren{u(t+h) - u(t)} - \paren{\abs{u}^{2\sigma}u
          -iu_{xx}}}_{L^2}\\
      &\leq \norm{h^{-1}\paren{e^{ih \partial_x^2} - I -
          ih \partial_x^2}
        u(t)}_{L^2}\\
      &\quad + \norm{h^{-1} \int_t^{t+h}e^{i\partial_x^2(t+h -s) }
        \abs{u}^{2\sigma}u_x ds - \abs{u}^{2\sigma} u_x }_{L^2}
    \end{split}
  \end{equation*}
  Since $s\mapsto (\abs{u}^{2\sigma}u_x)(s)$ is a continuous mapping
  into $L^2$, so is $s\mapsto e^{i\partial_x^2(t+h -s)}
  (\abs{u}^{2\sigma}u_x)(s)$.  Therefore, by the mean value theorem
  for integrals,
  \begin{equation*}
    h^{-1} \int_t^{t+h}e^{i\partial_x^2(t+h -s) }
    \abs{u}^{2\sigma}u_x ds = e^{i\partial_x^2(t+h -\tilde t)} (\abs{u}^{2\sigma}u_x)(\tilde t)
  \end{equation*}
  for some $t\leq \tilde t \leq t+h$.  Hence,
  \begin{equation*}
    \begin{split}
      & \norm{h^{-1}\paren{u(t+h) - u(t)} - \paren{\abs{u}^{2\sigma}u
          -iu_{xx}}}_{L^2}\\
      &\leq \norm{h^{-1}\paren{e^{ih \partial_x^2} - I -
          ih \partial_x^2}u(t)}_{L^2} + \norm{e^{i\partial_x^2(t+h
          -\tilde t)} (\abs{u}^{2\sigma}u_x)(\tilde t) -
        (\abs{u}^{2\sigma}u_x)(t) }_{L^2}.
    \end{split}
  \end{equation*}
  By the continuity of $u$ and the properties of the Schr\"odinger
  semigroup, this vanishes as $h\to 0$.  Thus, $u_t$ exists and $u_t =
  i u_{xx} - \abs{u}^{2\sigma} u_x$.  Again, because of the regularity
  of $u$, $u_t \in L^{\infty}(0,T;L^2)$.
\end{proof}
We note that when we prove below that $u\in C(0,T;H^{2}),$ this will
demonstrate that in fact $u_{t}\in C(0,T;L^{2}).$

\begin{lemma}
  \label{l:uniqueness}
  If $u \in L^\infty(0,T;H^2)\cap C(0,T;H^1)$ is mild solution, it is
  unique.
\end{lemma}

\begin{proof}
  This follows a Gronwall inequality, and Lemma \ref{l:c1} which says
  $u$ solves \eqref{e:gdnls} in the strong sense.  Indeed, if $u$ and
  $v$ are two solutions, then
  \begin{equation*}
    \begin{split}
      \frac{d}{dt}\norm{u-v}_{L^2}^2& = \int (\bar u - \bar v) (u_t -
      v_t) + \cc\\
      &  = \int (\bar u - \bar v) (-i u_{xx} +i v_{xx}) +\cc\\
      &\quad + \int (\bar u - \bar v) (\abs{u}^{2\sigma} u_x -
      \abs{v}^{2\sigma} v_x) +\cc
    \end{split}
  \end{equation*}
  The first integral vanishes after integrating by parts and adding
  the complex conjugate.  The second integral will be bounded as
  follows:
  \begin{equation*}
    \begin{split}
      &\int (\bar u - \bar v) (\abs{u}^{2\sigma} u_x -
      \abs{v}^{2\sigma} v_x) +\cc\\
      & = \int (\bar u - \bar v) (\abs{u}^{2\sigma} -
      \abs{v}^{2\sigma})u_x + \cc + \int (\bar u - \bar v)
      \abs{v}^{2\sigma} (u_x -
      v_x) + \cc\\
      &\leq \int 4 \sigma\max\set{\abs{u},\abs{v}}^{2\sigma -1}
      \abs{u-v}^2 \abs{u_x} +\int 2\sigma
      \abs{v}^{2\sigma-1}\abs{v_x}\abs{u-v}^2 \\
      &\leq C(\norm{u}_{L_t^\infty H^2_x}, \norm{v}_{L_t^\infty
        H^2_x}) \norm{u-v}_{L^2}^2
    \end{split}
  \end{equation*}
  Therefore, by Gronwall
  \begin{equation*}
    \norm{(u-v)(\cdot,t)}_{L^2}^2 \leq  \norm{(u-v)(\cdot,0)}_{L^2}^2 e^{C t}.
  \end{equation*}
  If the solutions agree at $t=0$, they agree for all time.
\end{proof}

\begin{cor}
  \label{l:contdep}
  If $u,v \in L^\infty(0,T;H^2)\cap C(0,T;H^1)$ are two mild
  solutions, then
  \begin{equation*}
    \norm{(u-v)(\cdot,t)}_{H^{s'}} \leq K e^{C t} \norm{(u-v)(\cdot,0)}_{L^2}^{1-s'/2}
  \end{equation*}
  for any $0\leq s'<2$.
\end{cor}
\begin{proof}
  Following the same scheme as in Lemma \ref{l:uniqueness}, we
  immediately obtain
  \begin{equation*}
    \norm{(u-v)(\cdot,t)}_{L^2} \leq e^{C t} \norm{(u-v)(\cdot,0)}_{L^2}.
  \end{equation*}
  Using the interpolation estimate, Lemma \ref{interp},
  \begin{equation*}
    \begin{split}
      \norm{(u-v)(\cdot,t)}_{H^{s'}}&\leq c
      \norm{(u-v)(\cdot,t)}_{H^{2}}^{s'/2}\norm{(u-v)(\cdot,t)}_{L^2}^{1-s'/2}\\
      & \leq c \paren{\norm{u}_{L^\infty_t H^2_x}
        +\norm{v}_{L^\infty_t H^2_x}}^{s'/2} e^{C(1-s'/2) t}
      \norm{(u-v)(\cdot,0)}_{L^2}^{1-s'/2}
    \end{split}
  \end{equation*}
\end{proof}

Having proved existence, uniqueness, and continuous dependence of
solutions, we will conclude this section with a little more on the
regularity of the $H^{2}$ solutions.  In particular, we will now prove
that the solutions are in $C([0,T];H^{2});$ proving this will complete
the proof of Theorem \ref{t:h2}.  We are attempting to show that for
any $t\in[0,T],$
$$\lim_{s\rightarrow t}\|u(\cdot,s)-u(\cdot,t)\|_{H^{2}}^{2}=0,$$
and the limit is understood to be one sided at $t=0$ and $t=T$.  We
can rewrite this as
$$\lim_{s\rightarrow t}\|u(\cdot,s)\|_{H^{2}}^{2}-2
\lim_{s\rightarrow t}\langle u(\cdot,s), u(\cdot,t)\rangle_{H^{2}} =
-\|u(\cdot,t)\|_{H^{2}}^{2}.$$ This will follow from the following two
items:
\begin{equation}\label{normConv}
  \lim_{s\rightarrow t}\|u(\cdot,s)\|_{H^{2}}^{2}=\|u(\cdot,t)\|_{H^{2}}^{2},
\end{equation}
\begin{equation}\label{weakConv}
  \lim_{s\rightarrow t}\langle u(\cdot,s), u(\cdot,t)\rangle_{H^{2}} = \|u(\cdot,t)\|_{H^{2}}^{2}.
\end{equation}
Of these, (\ref{weakConv}) follows from Theorem
\ref{weakTheoremContTime} with $k=2.$

To establish \eqref{normConv}, we begin by proving that the norm is
right-continuous at $t=0:$
$$\lim_{t\rightarrow 0^{+}}\|u(\cdot,t)\|_{H^{2}}^{2}=\|u_{0}\|_{H^{2}}^{2}.$$
Note that we have already shown that for all $\tau\in[0,T),$
$u(\cdot,t)\overset{{H^{2}}}{\rightharpoonup} u(\cdot,\tau)$ as
$t\rightarrow\tau^{+}.$ This implies
\begin{equation}\label{tauLimInf}
  \|u(\cdot,\tau)\|_{H^{2}}^{2}\leq \liminf_{t\rightarrow\tau^{+}}\|u(\cdot,t)\|_{H^{2}}^{2},\qquad \forall \tau\in[0,T).
\end{equation}
Evaluating \eqref{tauLimInf} at $\tau=0,$
$$\|u_{0}\|_{H^{2}}^{2}\leq \liminf_{t\rightarrow 0^{+}}\|u(\cdot,t)\|_{H^{2}}^{2}.$$
To complete the argument, we will use the energy inequality to show
that
\begin{equation}\label{limitSuperior}
  \limsup_{t\rightarrow 0^{+}}\|u(\cdot,t)\|_{H^{2}}^{2}\leq \|u_{0}\|_{H^{2}}^{2}.
\end{equation}
The energy inequality, established above, is
$$\frac{d}{dt}\|u_{\varepsilon}\|_{H^{2}}^{2}\leq c(\|u_{\varepsilon}\|_{H^{2}}^{2})^{\sigma+1}.$$
As above, we let $K$ be an upper bound for
$\|u_{\varepsilon}\|_{H^{2}}^{2},$ for all $\varepsilon\in(0,1],$ on
the time interval $[0,T].$ Then, the energy inequality becomes
$$\frac{d}{dt}\|u_{\varepsilon}\|_{H^{2}}^{2}\leq cK^{\sigma+1}.$$
Integrating this with respect to time, we get
$$\|u_{\varepsilon}(\cdot,t)\|_{H^{2}}^{2}-\|u_{0}\|_{H^{2}}^{2}\leq cK^{\sigma+1}t.$$
As we have established $u_{\varepsilon}\overset{H^2}{\rightharpoonup}
u$ for almost every time, we get
\begin{equation}\label{forLimSupAE}
  \|u(\cdot,t)\|_{H^{2}}^{2}-\|u_{0}\|_{H^{2}}^{2}\leq cK^{\sigma+1}t,\qquad \mathrm{a.e.\ }t\in[0,T).
\end{equation}
So \eqref{forLimSupAE} may not hold for values of $t$ in a set of
measure zero.  Now, let $\tau\in[0,T).$ We use (\ref{tauLimInf});
since a set of measure zero contains no interval, there must be a
sequence of times approaching $\tau$ from above for which we can use
the inequality of (\ref{forLimSupAE}).  This implies
$$\|u(\cdot,\tau)\|_{H^{2}}^{2}\leq \liminf_{t\rightarrow\tau^{+}}
\left\{\|u_{0}\|_{H^{2}}^{2}+cK^{\sigma+1}t\right\}=\|u_{0}\|_{H^{2}}^{2}+cK^{\sigma+1}\tau.$$
We see then that the inequality of (\ref{forLimSupAE}) actually holds
for all $t\in[0,T).$ Taking the limit superior as $t\rightarrow 0^{+}$
of both sides of the inequality of (\ref{forLimSupAE}), we conclude
that (\ref{limitSuperior}) holds.

This argument can then be repeated to demonstrate right-continuity of
the norm at any $t\in(0,T).$ This implies right-continuity of the norm
on $[0,T);$ we then need to establish left-continuity of the norm on
$(0,T].$ This, however, is actually the same argument, if we reverse
time.  We now briefly discuss left-continuity at $t=T,$ but
left-continuity at times in $(0,T)$ is similar.  Let $s$ be a
time-like variable, and consider the initial value problem
$$-iv_{s}+i|v|^{2\sigma}v_{x}+v_{xx}=0,\qquad v(\cdot,0)=u(T).$$  The same
estimates can be made for $v$ as we have made for $u,$ to find
existence of a solution $v$ on some time interval.  We can then repeat
the above argument to show that the norm of this solution is
right-continuous at $s=0.$ By uniqueness of solutions, though, we must
have $v(\cdot,s)=u(\cdot,T-s).$ This implies that the norm of $u$ is
left-continuous at $t=T.$ This complete the proof of Theorem
\ref{t:h2}.

As a final note on our $H^{2}$ solutions, we remark that they conserve
the Hamiltonian.
\begin{lemma}
  \label{l:hamiltonian}
  If $u \in C(0,T;H^2)\cap C(0,T;H^1)$ is mild solution, it conserves
  the Hamiltonian, \eqref{e:hamiltonian}.
\end{lemma}
\begin{proof}
  This follows by direct calculation, using Lemma \ref{l:c1}, which
  shows that the equality in \eqref{e:gdnls} holds in the sense of
  $L^2$ for almost all $t$.
\end{proof}

\section{Discussion}

We have demonstrated the existence of local-in-time mild solutions in
$H^1$, along with local in time existence and uniqueness in $H^2$.
Taking the initial data in $H^2$ also allows us to obtain continuous
dependence on the data, and continuity in time of solutions in $H^2$.
We were, unfortunately, unable to show that the $H^1$ solutions
belonged to $C(0,T;H^1).$ Our $H^2$ solutions are a step towards the
justifying the time dependent simulations appearing in
\cite{Liu:2013cq,Liu:2013ej}.

\subsection{Remarks on Our Construction}

An interesting feature of our $H^1$ result is the use of the
Hamiltonian in the construction of the energy functional,
\eqref{e:energy}.  Indeed, the $H^1$ norm is inadequate for the energy
method as $\frac{d}{dt}\|u_{x}\|_{L^{2}}^{2},$ after integration by
parts, is cubic in $u_{x},$ and thus cannot be bounded in terms of the
$H^1$ norm.  We are able to deal with this by using the conservation
of the Hamiltonian.  While control of the Hamiltonian itself does not
give control of a norm, we are able to add lower-order terms to it to
be able to control the $H^{1}$ norm, and we are able to estimate the
growth of these lower-order terms.

We are not able to prove uniqueness of the $H^{1}$ solutions we
construct.  We are aware of two main approaches to proving uniqueness
of solutions.  Uniqueness can come from the process of construction of
solutions, if the method used is, for instance, to prove existence of
a fixed point via the contraction mapping principle.  Otherwise, we
can make an estimate for the difference of two solutions (this is what
we did for our $H^{2}$ uniqueness theorem).  This estimate, however,
requires the equation (\ref{e:gdnls}) to be satisfied in a strong
sense.  Since our $H^{1}$ solutions are only mild solutions, this
approach is also not accessible to us; we were able to use it for our
$H^{2}$ solutions because of Lemma \ref{l:c1}.

\subsection{Implications of Greater Regularity in Time}
\label{towardsGlobal}

Were our $H^1$ solutions continuous in time into $H^1$, two further
things could be accomplished.  First, it is a prerequisite for the
notion of weak solution applied in
\cite{Grillakis:1987hj,Grillakis:1990jv}; a weak notion of
differentiability, \eqref{e:weak_soln}, is also required.  Thus, the
existence framework needed to fully justify the results of
\cite{Liu:2013cq} remains unresolved.

Second, if we did have continuity in time into $H^1$, we would be able
to conclude global-in-time existence for solutions with sufficiently
small data.  In this case, an {\it a priori} bound could be
constructed using the mass and Hamiltonian.  To demonstrate this, we begin with the
following estimate:
\begin{equation}
  \label{e:invariant_bound}
  \begin{split}
    \mathcal{M}+\mathcal{H}&\geq \norm{u}_{H^1}^2
    -\frac{1}{\sigma+1}\int
    \abs{u}^{2\sigma+1} \abs{u_x}\\
    &\geq \norm{u}_{H^1}^2 -\frac{1}{\sigma+1} \norm{u}_{L^{4\sigma+2}}^{2\sigma+1}\norm{u_x}_{L^2}\\
    &\geq \norm{u}_{H^1}^2 - c_\sigma \norm{u}_{H^1}^{2\sigma+2}\equiv
    f_{\sigma}\paren{ \norm{u}_{H^1}},
  \end{split}
\end{equation}
where $c_\sigma$ is a constant that depends only on $\sigma$.  For
$\sigma >0$, $f_{\sigma}$ has a local minimum at $0$ and a local
maximum at some $x_\star>0$.

To proceed, we assume that there exist mild solutions of the equation
such that:
\begin{itemize}
\item The mapping $t\mapsto \norm{u(t)}_{H^1}$ is continuous;
\item The solution conserves the invariants for the lifetime of the
  solution.
\end{itemize}

We call the above assumptions {\bf (A1)} and {\bf (A2)}, respectively.
Under these assumptions, the flow admits the following dichotomy:
\begin{lemma}
  \label{l:dichotomy}
  For $\sigma \geq 1$, such a flow admits the dichotomy that if
  $\mathcal{H} + \mathcal{M} < f_\sigma (x_\star),$ then:
  \begin{itemize}
  \item If $\norm{u_0}_{H^1}<x_\star$, then $\norm{u_0}_{H^1}<x_\star$
    for the lifetime of the solution.
  \item If $\norm{u_0}_{H^1}>x_\star$, then $\norm{u_0}_{H^1}>x_\star$
    for the lifetime of the solution.
  \end{itemize}
\end{lemma}
\begin{proof}
  By the assumption and \eqref{e:invariant_bound}, for the lifetime of
  the solution,
  \begin{equation}
    \label{e:inv_bound2}
    f_\sigma(\norm{u(t)}_{H^1}) \leq 
    \mathcal{H} + \mathcal{M} < f_\sigma(x_\star)
  \end{equation}
  Now, assume $\norm{u_0}_{H^1}< x_\star$ and let
  \[
  \tau \equiv \inf\set{t\leq T\mid \norm{u(t)}_{H^1}\geq x_\star},
  \]
  taking $\tau = +\infty$ if the set is empty.  If $\tau < \infty$,
  then by the continuity of $\norm{u(t)}_{H^1}$, there exists $0<t_1 <
  \tau \leq T$ at which $\norm{u(t_1)}_{H^1} = x_\star$.  But this
  violates \eqref{e:inv_bound2}, so $\norm{u(t)}_{H^1}< x_\star$ for
  the lifetime of the solution.  We have the analogous result if
  $\norm{u_0}_{H^1}> x_\star$.
\end{proof}
Consequently,
\begin{theorem}
  Assume {\bf (A1)} and {\bf (A2)}, and assume
  $\mathcal{H}+\mathcal{M}<f_{\sigma}(x_{\star})$ and
  $\|u_{0}\|_{H^{1}}<x_{\star}.$ Furthermore, assume that there is a
  lower bound for the time of existence of solutions, uniform in the
  $H^1$ size of the data.  Then the solution exists for all time.
\end{theorem}
We note that in the paper \cite{Tsutsumi:1980uy}, Tsutsumi and Fukuda
were able to conclude existence of global small solutions, but, as
noted in our introduction, their approach requires sufficiently smooth
nonlinearities.

The above results apply to our $H^2$ solutions in the following way.
Since our mild $H^2$ solutions are in $C(0,T;H^1)$ and conserve the
invariants, Lemma \ref{l:dichotomy} applies. Hence, with data in $H^2$
that is sufficiently small in $H^1,$ solutions remain uniformly
bounded in $H^1$ for the lifetime of the solution.  Moreover, since
our estimate of the time of existence depends on the $H^2$ size of the
data, we can infer that if the solution ceases to be global, it is
because $\norm{u_{xx}}_{L^2}$ is blowing up: there exists a finite $T_{\rm
  blowup} >0$ such that
\begin{equation}
  \limsup_{t\to T_{\rm blowup}^{-} }\norm{u_{xx}(t)}_{L^2} = \infty.
\end{equation}
It is our conjecture that no such blowup occurs, and that with a
refined analysis using the dispersive properties of the Schr\"odinger
semigroup, sufficiently small data in $H^2$ will yield solutions which
are global in time.

\clearpage

\bibliographystyle{abbrv} \bibliography{ASLWPDNLS}

\end{document}